\documentclass[final,onefignum,onetabnum,nohypdvips]{siamart251216}
\usepackage{xcolor}

\usepackage{braket,amsfonts}
\usepackage{amssymb,mathrsfs}
\usepackage{yhmath}
\usepackage{amsmath}
\usepackage{mathtools}
\usepackage{mathdots}
\allowdisplaybreaks
\usepackage{setspace}
\usepackage{array}
\usepackage{multirow}
\usepackage{soul}
\usepackage{amsmath}
\usepackage[
  letterpaper,
  right=1.2in,
  voffset=0pt
]{geometry}

\usepackage[caption=false]{subfig}
\usepackage{pgfplots}

\usepackage{mathtools}
\newsiamthm{claim}{Claim}
\newsiamthm{exmp}{Example}
\newsiamremark{remark}{Remark}
\newsiamremark{hypothesis}{Hypothesis}
\crefname{hypothesis}{Hypothesis}{Hypotheses}

\usepackage[T1]{fontenc}
\usepackage[utf8]{inputenc}

\usepackage{booktabs}

\usepackage{algorithmic}

\usepackage{graphicx,epstopdf}

\Crefname{ALC@unique}{Line}{Lines}

\usepackage{amsopn}

\usepackage{xspace}
\usepackage{bold-extra}
\usepackage[most]{tcolorbox}

\colorlet{texcscolor}{blue!50!black}
\colorlet{texemcolor}{red!70!black}
\colorlet{texpreamble}{red!70!black}
\colorlet{codebackground}{black!25!white!25}

\lstdefinestyle{siamlatex}{%
  style=tcblatex,
  texcsstyle=*\color{texcscolor},
  texcsstyle=[2]\color{texemcolor},
  keywordstyle=[2]\color{texemcolor},
  moretexcs={cref,Cref,maketitle,mathcal,text,headers,email,url},
}

\tcbset{%
  colframe=black!75!white!75,
  coltitle=white,
  colback=codebackground, 
  colbacklower=white, 
  fonttitle=\bfseries,
  arc=0pt,outer arc=0pt,
  top=1pt,bottom=1pt,left=1mm,right=1mm,middle=1mm,boxsep=1mm,
  leftrule=0.3mm,rightrule=0.3mm,toprule=0.3mm,bottomrule=0.3mm,
  listing options={style=siamlatex}
}

\newtcblisting[use counter=example]{example}[2][]{%
  title={Example~\thetcbcounter: #2},#1}

\newtcbinputlisting[use counter=example]{\examplefile}[3][]{%
  title={Example~\thetcbcounter: #2},listing file={#3},#1}

\DeclareTotalTCBox{\code}{ v O{} }
{ 
  fontupper=\ttfamily\color{black},
  nobeforeafter,
  tcbox raise base,
  colback=codebackground,colframe=white,
  top=0pt,bottom=0pt,left=0mm,right=0mm,
  leftrule=0pt,rightrule=0pt,toprule=0mm,bottomrule=0mm,
  boxsep=0.5mm,
  #2}{#1}

\patchcmd\newpage{\vfil}{}{}{}
\numberwithin{equation}{section}

\newcommand{\ds}{\mathrm ds}

\newcommand{\dtheta}{\mathrm d\theta}

\begin{tcbverbatimwrite}{tmp_\jobname_header.tex}
\title{Coefficient Recombination and Error Estimates for the L2 Method for Subdiffusion on Nonuniform Meshes\thanks{Submitted in the October of 2026.}}

\author{
Chaoyu Quan
\thanks{School of Science and Engineering, The Chinese University of Hong Kong (Shenzhen), Shenzhen 518172, China; Shenzhen International Center for Industrial and Applied Mathematics, Shenzhen Research Institute of Big Data, Shenzhen, 518000, China (\email{quanchaoyu@cuhk.edu.cn}).}
\and Shangyu Wang
\thanks{School of Science and Engineering, The Chinese University of Hong Kong, Shenzhen, 518172, Guangdong, People's Republic of China (\email{wangshangyu@link.cuhk.edu.cn}).}
\and Shijie Wang
\thanks{Shenshan Experimental School (Shenshan Branch of Shenzhen Second Senior High School), Shenzhen, 518200, China; Department of Mathematics, and SUSTech International Center for Mathematics, Southern University of Science and Technology, Shenzhen, China (\email{pariwong@163.com}).}
\and Jiang Yang
\thanks{Department of Mathematics, and SUSTech International Center for Mathematics, Southern University of Science and Technology, Shenzhen, China}
(\email{yangj7@sustech.edu.cn}).}

\headers{$L^2$-norm convergence of L2 method}{C. Quan, S. Wang, S. Wang, and J. Yang}
\end{tcbverbatimwrite}
\input{tmp_\jobname_header.tex}

\begin{document}
\maketitle
\sloppy

\begin{abstract}
The L2 discretization of the Caputo derivative is attractive for its high-order accuracy, but its analysis on nonuniform temporal meshes is complicated due to the lack of a monotone structure in the discrete coefficients. This difficulty is particularly significant for deriving sharp $L^2$ error estimates in the presence of the typical weak singularity of subdiffusion solutions near the initial time. In this work, we develop a coefficient-recombination framework for the L2 method on nonuniform meshes. Under the adjacent step-ratio condition
\(
0.86\le \rho_k\le 1.16,
\ k\ge2,
\)
we prove that the recombined L2 coefficients are positive and monotone.
A key ingredient is a cross-row recombination technique: with the recombination
parameter associated with a fixed time level, all coefficient rows retain positivity and monotonicity, which allows us to close an
induction for the $L^2$ error analysis.
For solutions exhibiting the standard initial weak singularity, we derive an $L^2$ error estimate on nonuniform meshes satisfying the step-ratio condition. On a modified graded mesh with grading parameter $r>1$, the estimate yields the convergence rate
\(
\mathcal{O}\!\left(
N^{-\min\{r\alpha,\,3-\alpha\}}
\right),
\)
and hence recovers the optimal order $\mathcal{O}(N^{-(3-\alpha)})$ whenever $r\ge(3-\alpha)/\alpha$.
Numerical experiments on random nonmonotone meshes and modified graded meshes confirm the theoretical estimates.
\end{abstract}


\begin{keywords}
Subdiffusion equation, L2 method, nonuniform meshes, coefficient recombination, graded meshes
\end{keywords}

\begin{AMS}
35R11, 65M15
\end{AMS}

\section{Introduction}

Time-fractional subdiffusion models, which can be derived from
continuous-time random walks \cite{gorenflo2002time,MK2000}, are widely used to describe processes
with memory effects and temporal nonlocality \cite{podlubny1998fractional,benson2000application,gafiychuk2008mathematical}. Consider the following subdiffusion equation:
\begin{equation}\label{model_subdiffusion}
\begin{aligned}
    \partial_t^\alpha u(t,x) & = \Delta u(t,x)+f(t,x),&& (t,x)\in (0,T] \times \Omega,\\
    u(t,x) &= 0,&& (t,x)\in (0,T] \times \partial\Omega,\\
    u(0,x)& = u^0(x),&& x\in \Omega,
\end{aligned}
\end{equation}
where $\Omega$ is a bounded Lipschitz domain in $\mathbb R^d$, $f(t,x) \in L^\infty([0,\infty);L^2(\Omega)) \cap BV([0,\infty); L^2(\Omega))$ is a temporal bounded variation function, $u^0\in H_0^1(\Omega)$, and $\partial_t^\alpha u(t,x)$ is Caputo derivative \cite{caputo1967linear} defined as
\begin{equation*}
    \partial_t^\alpha u \coloneqq \frac{1}{\Gamma(1-\alpha)} \int_0^t \frac{u'(s)}{(t-s)^\alpha} \, ds, \quad 0<\alpha<1.
\end{equation*}

Over the past decade, numerous numerical methods have been proposed and
analyzed to solve the time-fractional diffusion equation, many of which
utilize uniform time meshes. These include, for example, the L1 scheme
analyzed by Lin and Xu \cite{lin2007finite}, the L2-1$_\sigma$ formula
introduced by Alikhanov \cite{alikhanov2015new}, the L1-2 formula derived
by Gao, Sun, and Zhang \cite{gao2014new}, and the L2 scheme analyzed by
Lv and Xu \cite{lv2016error}. More recently, Alikhanov and Huang \cite{alikhanov2021high} proposed an L2 approximation and corresponding schemes for the subdiffusion equation with variable coefficients.

However, the exact solution to the time-fractional diffusion equation may exhibit low regularity near the initial time.
This lack of regularity can lead to a significant reduction in the accuracy of numerical methods.
To address this issue, the use of non-uniform meshes has been proposed.
In particular, the concept of graded meshes has emerged as an effective solution, ensuring that the computational results align with the theoretical accuracy.
Stynes,
O'Riordan, and Gracia \cite{stynes2017error} derived a sharp error
estimate for the L1 scheme on graded meshes, explicitly characterizing
the dependence of the convergence rate on the temporal regularity of
the solution and the mesh grading parameter. Kopteva
subsequently developed a more general analysis
of the L1 scheme in two and three spatial dimensions, covering both
finite difference and finite element spatial discretizations and error
estimates in the spatial \(L^\infty\) and \(L^2\) norms \cite{kopteva2019error} , and later
treated semilinear subdiffusion equations using upper and lower
solutions \cite{kopteva2020semilinear}.
For higher-order temporal discretizations, Chen and Stynes
\cite{chen2019error} analyzed the L2-\(1_\sigma\) formula on fitted
meshes combining a graded initial part with a quasiuniform part away
from \(t=0\), and established high-order convergence for both
initial-value and initial-boundary value problems.
Kopteva and Meng \cite{kopteva2020error} used barrier functions and the inverse
monotonicity of the discrete operators to derive sharp
pointwise-in-time error bounds for the L1 and L2-\(1_\sigma\) schemes
on quasi-graded meshes with arbitrary degrees of grading.
For the L2 scheme of order \(3-\alpha\), whose discrete
operator lacks an M-matrix structure, Kopteva \cite{kopteva2021error} represented
the operator as a product of two M-matrices under suitable mesh
conditions, thereby recovering the inverse monotonicity and allowing
the same technique to be applied.
More recently, Kelly and Kopteva \cite{kelly2026pointwise} further extended graded-mesh pointwise analysis to L1 discretizations of semilinear and quasilinear subdiffusion problems.

The analysis becomes more difficult on general nonuniform meshes
because the discrete Caputo coefficients depend on the current time
level, destroying the structure of uniform-mesh formulas.
For the nonuniform L1 scheme applied to linear reaction-subdiffusion
equations, Liao, Li, and Zhang \cite{liao2018sharp} established an
analytical framework based on complementary convolution kernels, a
discrete fractional Gr\"onwall inequality, and global consistency
analysis, obtaining sharp error estimates.
The scope of the nonuniform
L1 analysis was further broadened by Wang and Stynes
\cite{wang2023optimal}, who proved that the numerical solution
preserves the optimal long-time decay rate on a general class of
nonuniform meshes. Beyond the L1 scheme, Liao, McLean, and Zhang
\cite{liao2019discrete,liao2018second} extended the discrete
fractional Gr\"onwall framework to general discrete Caputo operators
and applied it to the nonuniform L2-\(1_\sigma\) scheme.  The L2 method is
more challenging to analyze because of the complex structure and lack
of monotonicity of its discrete coefficients. Quan and Wu
\cite{quan2022h} proved the positive semidefiniteness of a bilinear
form associated with the L2 operator under suitable restrictions on
adjacent time-step ratios and thereby established long-time
\(H^1\)-stability and corresponding \(H^1\)-error
estimates on general nonuniform meshes.


Against this background, we study the $L^2$-norm convergence of
the L2 method on nonuniform temporal meshes satisfying the
two-sided adjacent step-ratio condition $0.86 \le \rho_k := \tau_k/\tau_{k-1} \le 1.16$ for
$k\ge2$. The
\(H^1\)-error estimate in \cite{quan2022h} gives the rate
\((5-\alpha)/2\) on modified graded meshes and therefore does not
directly yield the optimal \(L^2\)-norm rate \(3-\alpha\). In contrast to the quasi-graded-mesh analysis in \cite{kopteva2021error}, our general $L^2$-error estimate applies to nonuniform meshes satisfying the above step-ratio condition.
Inspired by the coefficient-recombination technique introduced
by Lv and Xu \cite{lv2016error} for uniform meshes, we develop
a cross-row recombination technique for nonuniform meshes, which forms the main analytical contribution of this work.
At each time level, the discrete fractional derivative in the error
equation is a weighted sum of error increments, with coefficients
depending on the current time level. To estimate the error at an
arbitrary fixed time level, we choose the recombination parameter
associated with the current time level and use the same parameter to rewrite
all error equations up to this time level in terms of recombined errors.
Under the restriction on adjacent time-step ratios, we prove
that the recombined coefficients are positive and monotone.
The common error recombination allows the energy estimates at
different time levels to be combined in the induction. Together with the coefficient bounds, this yields \(L^2\)-error estimates for solutions with the typical initial weak singularity and recovers the optimal convergence rate \(3-\alpha\) on modified graded meshes with a suitable grading parameter.



The remainder of this paper is organized as follows. Section~\ref{sec_2}
recalls the L2 approximation of the Caputo derivative on nonuniform
temporal meshes and summarizes the basic properties of its
coefficients. Section~\ref{sec_3} develops the cross-row recombination
technique and establishes the positivity, monotonicity, and
coefficient inequalities needed for the energy analysis.
Based on these results, Section~\ref{sec_4} derives \(L^2\)-error estimates
on nonuniform meshes satisfying the step-ratio condition and
establishes the optimal convergence rate on suitably modified
graded meshes. Section~\ref{sec:numerical} tests the error estimate
on random meshes and the optimal convergence rate on modified
graded meshes, and compares the efficiency of nonuniform and
uniform time stepping.
Section~\ref{sect5} concludes the paper.
The appendices provide the technical estimates needed to establish
the positivity and monotonicity of the recombined coefficients.

Throughout the paper, \(|\cdot|_{\infty}\) denotes the vector
maximum norm, while \((\cdot,\cdot)\) and \(\|\cdot\|_0\) denote the
standard inner product and norm in \(L^2(\Omega)\), respectively.
Unless otherwise specified, \(C>0\) denotes a generic constant that may
vary from line to line. It may depend on \(T\), \(\alpha\), \(\Omega\), and the fixed bounds on the adjacent
time-step ratios.
On modified graded meshes, it may additionally depend on the
fixed parameters \(r\) and \(k_0\).

\section{Preliminaries on L2 approximation}\label{sec_2}
Take a nonuniform time mesh $0=t_0<t_1<\cdots<t_N=T$. Define the time steps and adjacent step ratios by
$\tau_j:=t_j-t_{j-1}$ for $1\le j\le N$
and \(
\rho_j:=\frac{\tau_j}{\tau_{j-1}}
\) for $ 2\le j\le N$, respectively.

Following \cite{lv2016error}, let $H_1^1$ be the linear interpolant
and $H_2^j$ the quadratic interpolant determined by
\begin{equation}\label{eq:H2j}
H_1^1(t_\ell)=u^\ell,\quad \ell=0,1,
\qquad
H_2^j(t_\ell)=u^\ell,\quad \ell=j-1,j,j+1.
\end{equation}
For $k=1$, we use $H_1^1$ on $[t_0,t_1]$.
For $k\ge2$, we use $H_2^j$ on $[t_{j-1},t_j]$ for
$1\le j\le k-1$, and $H_2^{k-1}$ also on $[t_{k-1},t_k]$.
Integrating the derivative of this interpolant against the Caputo
kernel gives the L2 operator
\begin{equation}\label{eq:Lk1}
\begin{aligned}
L_1^\alpha u& \coloneqq \frac{1}{\Gamma(2-\alpha)\tau_1^{\alpha}}\delta_1 u,\\
    L_k^\alpha u
  &\coloneqq\frac{1}{\Gamma(1-\alpha)}\left( ( c_{k}^{(k)} + c_{k-1}^{(k)}) \delta_k u- a_{k}^{(k)}\delta_{k-1} u - a_{1}^{(k)}\delta_1 u  +\sum_{j=2}^{k-1}  (c_{j-1}^{(k)}- a_{j}^{(k)}) \delta_j u \right), \quad k\geq 2,
    \end{aligned} 
\end{equation}
where $\delta_j u \coloneqq u^j-u^{j-1}$, and the L2 coefficients $a_{j}^{(k)}$ and $c_{j}^{(k)}$ are computed by
\begin{equation}
\begin{aligned}\label{eq:aj_cj}
a_{j}^{(k)} &  = \int_0^1 \frac{-2 \tau_j(1-\theta)-\tau_{j+1}}{(\tau_{j}+\tau_{j+1})(t_k-(t_{j-1}+\theta \tau_j))^\alpha}\,\dtheta,\qquad 1\leq j\leq k-1,\\
 c_{j}^{(k)}  &= \int_0^1 \frac{\tau_j^2(2\theta-1)}{\tau_{j+1}(\tau_{j}+\tau_{j+1})(t_k-(t_{j-1}+\theta \tau_j))^\alpha}\,\dtheta,\qquad 1\leq j\leq k-1
\end{aligned}
\end{equation}
and
\begin{equation}
\begin{aligned}
    a_{k}^{(k)} & = \int_0^1 \frac{\tau_k^2(2\theta-1)}{\tau_{k-1}(\tau_{k-1}+\tau_{k})(t_k-(t_{k-1}+\theta \tau_k))^\alpha}\dtheta,\\
    c_{k}^{(k)} & = \int_0^1 \frac{2\tau_k\theta+\tau_{k-1}}{(\tau_{k-1}+\tau_{k})(t_k-(t_{k-1}+\theta \tau_k))^\alpha}\,\dtheta.
\end{aligned}
\end{equation}
Explicit evaluations of these integrals are given in
\cite{quan2022h}. In particular, the terminal coefficients satisfy
\begin{equation}\label{eq:a_k1}
\begin{aligned}
    a_{k}^{(k)}
    &=\frac{\alpha\tau_k^2}{(2-\alpha)(1-\alpha)\tau_{k-1}(\tau_{k-1}+\tau_{k}) \tau_k^\alpha},\\
   c_{k}^{(k)}
    &=\frac{1}{(1-\alpha)\tau_k^\alpha}+\frac{\alpha\tau_k}{(2-\alpha)(1-\alpha)(\tau_{k-1}+\tau_{k}) \tau_k^\alpha}.
    \end{aligned}
\end{equation}
As shown in \cite{quan2022h}, the L2 coefficients have the properties: $ a_{j}^{(k)}<0$, $ c_{j}^{(k)}>0$ for $1\leq j\leq k-1$, and $a_{k}^{(k)}>0$, $ c_{k}^{(k)}>0$.

\section{Recombination of L2 operator and its properties}\label{sec_3}

Given a nonuniform mesh $\{\tau_k\}_{1\leq k\leq N}$ with $N$ time steps, the L2 scheme for the subdiffusion equation \eqref{model_subdiffusion} is
\begin{equation}\label{eq:sch_sub_graded}
\begin{aligned}
    L_k^\alpha u 
    &= \Delta u^k+f^k,
\end{aligned}    
\end{equation}
where $f^k=f(t_k,x)$.
The error equation associated with scheme \eqref{eq:sch_sub_graded} is then given by
\begin{equation}\label{error_eq}
L_k^\alpha e = \Delta e^k - r^k, \quad 1 \leq k \leq N,
\end{equation}
where $e^k \coloneqq u(t_k) - u^k$ is the error at time point $t_k$ and the truncation error $r^k$ is given by 
\begin{equation}\label{eq:localrk}
r^k \coloneqq \frac{1}{\Gamma(1-\alpha)}\int _0^{t_k}(t_k-s)^{-\alpha}\partial_s[u(s)-I_2u(s)]\,\ds,\quad k\ge 1.
\end{equation} 
Here, for $k=1$, $I_2u$ is the linear interpolant of $u(t_0)$ and $u(t_1)$ on $(t_0,t_1)$.
For $k\geq2$, $I_2 u=H_{2}^j u$ on $(t_{j-1},t_j)$ for $1\leq j\leq k-1$ and $I_2 u=H_{2}^{k-1} u$ on $(t_{k-1},t_k)$.
$H_2^j u(t)$ represents the quadratic interpolating polynomial obtained by replacing the numerical solution values $\{u^{j-1}, u^j, u^{j+1}\}$ in \eqref{eq:H2j} with the exact solution values $\{u(t_{j-1}), u(t_j), u(t_{j+1})\}$.
From the definition of L2 operator \eqref{eq:Lk1}, we can write
\begin{equation}\label{def:L2_0}
L_k^\alpha e = \frac{1}{\Gamma(1-\alpha)}\sum_{j=1}^k {m}_{k, j} \delta_j e, \qquad 1 \leq k \leq N,
\end{equation}
where $\delta_j e = {e}^j - {e}^{j-1}$ and $m_{k,j}$ are associated with the L2 coefficients.
As a consequence, the error equation \eqref{error_eq} can be written in the matrix form of
\begin{equation}
\frac{1}{\Gamma(1-\alpha)}\mathbf{M}
 \begin{pmatrix}
\delta_1 e \\
\delta_2 e \\
\vdots \\
\delta_N e
\end{pmatrix}
 = 
 \Delta 
 \begin{pmatrix}
e^1 \\
e^2 \\
\vdots \\
e^N
\end{pmatrix}
-
\begin{pmatrix}
r^1 \\
r^2 \\
\vdots \\
r^N
\end{pmatrix},
\end{equation}
where the coefficient matrix
{\footnotesize
\begin{equation}\label{matM}
\mathbf M \coloneqq (m_{k,j})_{1\leq k,j\leq N}=
\begin{pmatrix}
(1-\alpha)^{-1}\tau_1^{-\alpha} & \\
-a_1^{(2)}-a_2^{(2)}& c^{(2)}_1+c^{(2)}_2 \\
-a_1^{(3)}& d^{(3)}_2-a_3^{(3)}& c^{(3)}_2+c^{(3)}_3 \\
-a_1^{(4)} & d^{(4)}_2 & d_3^{(4)}-a_4^{(4)}
& c^{(4)}_3+c^{(4)}_4\\
\vdots& \vdots & \ddots &\ddots & \ddots\\
-a_1^{(N)}& d^{(N)}_2&\cdots& d^{(N)}_{N-2}
& d^{(N)}_{N-1}-a^{(N)}_{N}&c^{(N)}_{N-1}+c^{(N)}_N
\end{pmatrix},
\end{equation}
}
with
\begin{equation}
\label{eq:d}    
 d_{j}^{(k)}\coloneqq c_{j-1}^{(k)}- a_{j}^{(k)}.
\end{equation}
Note that $d_{j}^{(k)} > 0$
for $2\le j\le k-1 \leq N-1$ (see \cite{quan2022h} for the positivity of $d_j^{(k)}$). Moreover, we have the following lemma on the positivity and monotonicity of $m_{k,j}$.
\begin{lemma}\label{origin_matrix_row}(see \cite{quan2022h})
The coefficient matrix in \eqref{matM} satisfies the following properties:
\begin{itemize}
\item[(i)]$m_{k,j}>0$ for all $1\leq j\leq k $ and $j \not = k-1$; 
\item [(ii)] If the nonuniform mesh $\{\tau_j\}_{1\leq j\leq N}$ satisfies
\begin{equation}\label{condition:rho}
    \frac{1}{\rho_{j+1}}\ge\frac{1}{\rho_{j}^2(1+\rho_{j})}-3\quad
 \mbox{with}\quad \rho_j \coloneqq \tau_j/\tau_{j-1}, ~ 2\leq j\leq N-1,
\end{equation}
then
\begin{equation}
m_{k,j} \leq m_{k,j+1}, \quad \text{for all } 1\leq j \leq k-3, ~4\leq k\leq N. 
\end{equation}
\end{itemize}
\end{lemma}

\subsection{Recombination of L2 coefficients}
One main difficulty in analyzing the error equation \eqref{error_eq} of the L2 method is that in the cumulative form of $\delta_je$, the L2 coefficients are not monotonic. 
To overcome this obstacle, recombination techniques are proposed. 
In the case of uniform meshes, Lv and Xu employ a recombination technique by introducing a parameter defined as a negative half of the ratio between the coefficients of $e^{k-1}$ and $e^k$ \cite{lv2016error}.
Under the assumption that the exact solution to the subdiffusion equation is three times continuously differentiable in time, the optimal $L^2$-norm convergence of order $3-\alpha$ is achieved for the L2 method with uniform meshes. However, it is well-known that this regularity assumption is not satisfied even for smooth initial data. 
Later, to relax the regularity assumption, Kopteva \cite{kopteva2021error} employed a similar recombination technique to derive the inverse-monotonicity of the discrete fractional-derivative operator and achieve the pointwise optimal convergence of order $3-\alpha$ on quasi-graded meshes.

In the case of nonuniform meshes, we introduce the same recombination parameter as \cite{lv2016error}:
\begin{equation}\label{eta}
\eta_1 := 0, \qquad
{\eta_k} 
\coloneqq \frac12 \left(1-\frac{m_{k,k-1}}{m_{k,k}}\right), \quad 2 \leq k \leq N.
\end{equation}
Consequently, the L2 operator in \eqref{def:L2_0} could be reformulated as
\begin{equation}\label{new_scheme}
L_k^\alpha {e} = \frac{1}{\Gamma(1-\alpha)}\sum_{j=1}^k \bar{m}_{k, j} (\bar{e}^j - \bar{e}^{j-1}), \quad 1 \leq k \leq N,
\end{equation}
where $\bar{e}^j \coloneqq e^j - {\eta_k} e^{j-1},~1\leq j\leq k$ and $\bar e^0 \coloneqq
e^0 = 0$.

Successive substitution of $e^j=\bar e^j+\eta_ke^{j-1}$ into
\eqref{def:L2_0} yields
\begin{equation}\label{def_bar_m}
\bar{m}_{k, j}= 
\sum_{i=j}^{k} m_{k,i}\eta_k^{i-j},\quad 1\leq j\leq k.
\end{equation}


\subsection{Cross-row recombination technique}
\label{sec:cross-row-technique}
We first establish some bounds for the recombined parameter $\eta$, and a lower bound for the normalized recombined coefficient difference
$(\bar m_{n,n-1}-\bar m_{n,n-2})/m_{n,n}$,
with $3\le n\le N$.
We then introduce the cross-row recombination technique and
prove the properties of the resulting coefficients, which will be used in Section~\ref{sec_4}.
Set $\rho_*=0.86$ and $\rho^*=1.16$.

\begin{lemma}\label{lm:eta_prop}
Assume that $\rho_j\in[\rho_*,\rho^*]$ for $2\le j\le N$. The parameters defined in \eqref{eta} satisfy
\begin{equation}\label{eq:eta_uniform_bound}
0<\eta_j\le\eta_*\coloneqq\frac{3+2\sqrt2}{2+4\sqrt2}<1,
\qquad 2\le j\le N.
\end{equation}
Moreover,
\begin{equation}\label{eq:eta_variation_bound}
|\eta_k-\eta_n|\le\frac{4\alpha}{25},
\qquad 3\le n\le k\le N.
\end{equation}
\end{lemma}
\begin{proof}
The bound \eqref{eq:eta_uniform_bound} follows from
\eqref{App-eq:eta-uniform-main} in the appendix.
For $j\ge3$, the lower and upper bounds in
\eqref{App-eq:eta-band-main} are independent of $j$ and differ by
$4\alpha/25$, which proves \eqref{eq:eta_variation_bound}.
\end{proof}

We next estimate the difference between the two coefficients
preceding the diagonal in a recombined row.
\begin{lemma}\label{feasible_reg_2}
Assume that $\rho_j\in[\rho_*,\rho^*]$ for $2\le j\le N$. Then the recombined coefficients defined in \eqref{def_bar_m} satisfy
\begin{equation}\label{eq:sigma_lower_bound}
\frac{\bar m_{n,n-1}-\bar m_{n,n-2}}{m_{n,n}}
\ge\frac{\alpha}{25},
\qquad 3\le n\le N.
\end{equation}
\end{lemma}
\begin{proof}
For $3\le n\le N$, set
\begin{equation}\label{eq:sigma_definition}
\sigma_n\coloneqq
\frac{\bar m_{n,n-1}-\bar m_{n,n-2}}{m_{n,n}}.
\end{equation}
Define the normalized coefficients
\begin{align}
r_1^{(n)}(\alpha)&\coloneqq\frac{m_{n,n-1}}{m_{n,n}}
=1-2\eta_n,\label{eq:r1}\\
r_2^{(n)}(\alpha)&\coloneqq\frac{m_{n,n-2}}{m_{n,n}}.
\label{eq:r2}
\end{align}
For $n\ge4$, using \eqref{def_bar_m} and \eqref{eq:r1}, we obtain
\begin{equation}\label{eq:sigma_square_identity}
4\sigma_n=\bigl(r_1^{(n)}(\alpha)+1\bigr)^2-4r_2^{(n)}(\alpha)
=\bigl(r_1^{(n)}(\alpha)-1+\alpha\bigr)^2
+F(\alpha,\boldsymbol\rho),
\end{equation}
where, for $n\ge4$, $\boldsymbol\rho=(\rho_n,\rho_{n-1},\rho_{n-2})$ and
\begin{equation}\label{eq:F_definition}
F(\alpha,\boldsymbol\rho)
\coloneqq(4-2\alpha)r_1^{(n)}(\alpha)-4r_2^{(n)}(\alpha)
+2\alpha-\alpha^2.
\end{equation}
Using the continuous endpoint values of these ratios, we have
\begin{equation}\label{eq:F_endpoints}
F(0,\boldsymbol\rho)=0,
\qquad
F(1,\boldsymbol\rho)=\frac{-2\rho_n^2+2\rho_n+1}{2\rho_n+1}
\ge\frac{3011}{16700}>\frac4{25}.
\end{equation}
Since $F$ is concave in $\alpha$ by
Appendix~\ref{App-sec:concavity}, it follows that
\begin{equation}\label{eq:F_lower_bound}
F(\alpha,\boldsymbol\rho)
\ge(1-\alpha)F(0,\boldsymbol\rho)+\alpha F(1,\boldsymbol\rho)
\ge\frac{4\alpha}{25}.
\end{equation}
Thus $\sigma_n\ge\alpha/25$ for $n\ge4$.
For $n=3$, $4\sigma_3=(r_1^{(3)}-1+\alpha)^2+F_3$, where
$F_3\coloneqq(4-2\alpha)r_1^{(3)}-4r_2^{(3)}+2\alpha-\alpha^2$.
The bound $F_3\ge4\alpha/25$ is proved in
Appendix~\ref{App-sec:bernstein-certificates}.
Thus $\sigma_3\ge\alpha/25$.
\end{proof}

In \eqref{def_bar_m}, the coefficients at time level $t_n$
are recombined using $\eta_n$.
We now extend this construction to all preceding time levels
using the recombination parameter associated with the current
time level. We refer to this as the cross-row recombination technique.
Fix an arbitrary $k$ with $3\le k\le N$ and define
\begin{equation}\label{eq:cross_row_coefficients}
\bar m_{n,j}^{(k)}
\coloneqq
\sum_{\ell=j}^n m_{n,\ell}\eta_k^{\ell-j},
\qquad 1\le j\le n\le k.
\end{equation}
In particular,
$\bar m_{k,j}^{(k)}=\bar m_{k,j}$.
The following estimates provide the coefficient
bounds needed below.

\begin{lemma}\label{lem:cross_row_estimates}
Assume that $\rho_j\in[\rho_*,\rho^*]$ for $2\le j\le N$. Let $3\le k\le N$. Then, for every $2\le n\le k$, the coefficients defined in \eqref{eq:cross_row_coefficients} satisfy
\begin{subequations}\label{eq:cross_terminal_estimates}
\begin{align}
\frac{\bar m_{n,n-1}^{(k)}}{m_{n,n}}&\ge\frac15,
\label{eq:cross_terminal_positive}\\
\frac{\bar m_{n,n}^{(k)}-\bar m_{n,n-1}^{(k)}}{m_{n,n}}
&\ge\frac\alpha2,
\label{eq:cross_terminal_difference}\\
\frac{\bar m_{n,n}^{(k)}-\bar m_{n,n-1}^{(k)}}{m_{n,n}}
-\frac{m_{n-1,n-1}}{m_{n,n}}\eta_k^2
&\ge\frac\alpha{10}.
\label{eq:cross_terminal_gradient}
\end{align}
\end{subequations}
\end{lemma}
\begin{proof}
Since $m_{n,n-1}/m_{n,n}=1-2\eta_n$, we have
\begin{equation}\label{eq:cross_terminal_identities}
\frac{\bar m_{n,n-1}^{(k)}}{m_{n,n}}
=1-2\eta_n+\eta_k,
\qquad
\frac{\bar m_{n,n}^{(k)}-\bar m_{n,n-1}^{(k)}}{m_{n,n}}
=2\eta_n-\eta_k.
\end{equation}
Substituting these identities into
\eqref{App-eq:cross-L1}--\eqref{App-eq:cross-L4}, proved in
Appendix~\ref{App-sec:cross-row-lemma}, gives the assertions.
\end{proof}

We now establish the main result on the recombined L2 coefficients.
\begin{theorem}[Cross-row recombination]
\label{lem:cross_row_recombination}
Assume that the time mesh $0=t_0<t_1<\cdots<t_N=T$ satisfies
$\rho_j\in[0.86,1.16]$ for $j\ge2$.
For an arbitrary target level $3\le k\le N$, let
$\bar m_{n,j}^{(k)}$ be defined by \eqref{eq:cross_row_coefficients}.
Then, for every $2\le n\le k$, the following properties hold:
\begin{itemize}
\item[(i)] $\bar m_{n,j}^{(k)}>0$ for all $1\le j\le n$;
\item[(ii)] $\bar m_{n,j}^{(k)}\le\bar m_{n,j+1}^{(k)}$
for all $1\le j\le n-1$;
\item[(iii)] $\displaystyle
\sum_{j=1}^{n-1}\frac{\bar m_{n,j+1}^{(k)}-\bar m_{n,j}^{(k)}}{m_{n,n}}
=1-\frac{\bar m_{n,1}^{(k)}}{m_{n,n}}<1$;
\item[(iv)] $\displaystyle
\frac{\Gamma(1-\alpha)}{m_{n,n}}\eta_k^2
\le\frac{\bar m_{n,n}^{(k)}-\bar m_{n,n-1}^{(k)}}{m_{n,n}}
\frac{\Gamma(1-\alpha)}{m_{n-1,n-1}}$.
\end{itemize}
\end{theorem}
\begin{proof}
By definition, $\bar m_{n,n}^{(k)}=m_{n,n}$ and
\begin{equation}\label{eq:cross_row_backward}
\bar m_{n,j}^{(k)}=m_{n,j}+\eta_k\bar m_{n,j+1}^{(k)},
\qquad 1\le j\le n-1.
\end{equation}
Lemma~\ref{origin_matrix_row}(i) gives $m_{n,n}>0$ and
$m_{n,j}>0$ for $1\le j\le n-2$.
By \eqref{eq:cross_terminal_positive},
$\bar m_{n,n-1}^{(k)}>0$.
For $n\ge3$, applying \eqref{eq:cross_row_backward}
successively for $j=n-2,n-3,\ldots,1$, with $\eta_k>0$,
gives $\bar m_{n,j}^{(k)}>0$ for all these indices.
Together with $\bar m_{n,n}^{(k)}=m_{n,n}>0$,
this proves (i).

To prove (ii), the last difference is positive by
\eqref{eq:cross_terminal_difference}. For $n\ge3$, the preceding
difference satisfies
\begin{equation}\label{eq:cross_penultimate_identity}
\begin{aligned}
\frac{\bar m_{n,n-1}^{(k)}-\bar m_{n,n-2}^{(k)}}{m_{n,n}}
&=
\frac{\bar m_{n,n-1}-\bar m_{n,n-2}}{m_{n,n}}
-(\eta_k-\eta_n)^2\ge \frac{\alpha}{25}
-\frac{16\alpha^2}{625}
\ge \frac{9\alpha}{625}>0.
\end{aligned}
\end{equation}
The identity follows by expanding
\eqref{eq:cross_row_coefficients} and using $r_1^{(n)}=1-2\eta_n$;
the inequalities follow from Lemmas~\ref{lm:eta_prop} and
\ref{feasible_reg_2}. Subtracting adjacent instances of
\eqref{eq:cross_row_backward} gives
\begin{equation}\label{eq:cross_difference_recursion}
\begin{aligned}
\bar m_{n,j+1}^{(k)}-\bar m_{n,j}^{(k)}
={}&m_{n,j+1}-m_{n,j}
+\eta_k\bigl(\bar m_{n,j+2}^{(k)}-\bar m_{n,j+1}^{(k)}\bigr), \qquad1\le j\le n-3.
\end{aligned}
\end{equation}
The mesh condition \eqref{condition:rho} holds because
\(
\frac{1}{\rho_j^2(1+\rho_j)}-3
\le\frac{1}{(0.86)^2(1+0.86)}-3<0<\frac{1}{\rho_{j+1}}
\).
For $n\ge4$, Lemma~\ref{origin_matrix_row}(ii) gives
$m_{n,j+1}-m_{n,j}\ge0$ for $1\le j\le n-3$.
Since $\eta_k>0$, applying \eqref{eq:cross_difference_recursion}
successively for $j=n-3,n-4,\ldots,1$ yields
$\bar m_{n,j+1}^{(k)}-\bar m_{n,j}^{(k)}\ge0$
for all these indices.
This proves (ii), including $n=2$, for which there is
only the last difference.
Assertion (iii) follows by summing the differences and using
(i) and $\bar m_{n,n}^{(k)}=m_{n,n}$.
Finally, multiplying
\eqref{eq:cross_terminal_gradient} by
$\Gamma(1-\alpha)/m_{n-1,n-1}>0$ gives
\begin{equation}\label{eq:cross_gradient_margin}
\begin{aligned}
&\frac{\bar m_{n,n}^{(k)}-\bar m_{n,n-1}^{(k)}}{m_{n,n}}
\frac{\Gamma(1-\alpha)}{m_{n-1,n-1}}
-\frac{\Gamma(1-\alpha)}{m_{n,n}}\eta_k^2\ge\frac\alpha{10}\frac{\Gamma(1-\alpha)}{m_{n-1,n-1}}>0,
\end{aligned}
\end{equation}
which proves (iv).
\end{proof}

Taking $n=k$ in Theorem~\ref{lem:cross_row_recombination}
gives positivity and monotonicity of the coefficients $\bar m_{k,j}$
for $k\ge3$. The row $k=2$ follows directly from
\begin{equation}\label{eq:startup_own_row}
\bar m_{2,1}=(1-\eta_2)m_{2,2},\qquad
\bar m_{2,2}=m_{2,2}>0,
\end{equation}
and $0<\eta_2<1$ from Lemma~\ref{lm:eta_prop}.

The recombination parameter and the diagonal coefficients
also satisfy
\begin{equation}\label{eq:3.23_new}
m_{k,k}^{-1}\eta_k\le m_{k-1,k-1}^{-1},\qquad k\ge2.
\end{equation}
For $k\ge3$, setting $n=k$ in
\eqref{eq:cross_terminal_gradient} gives
$\eta_k-(m_{k-1,k-1}/m_{k,k})\eta_k^2\ge\alpha/10$ and then
division by $\eta_k>0$ proves \eqref{eq:3.23_new}.
For $k=2$, the coefficient definitions give
\begin{equation}
\label{eq:startup-original-coefficients}
\begin{aligned}
m_{2,1}&=-a_1^{(2)}-a_2^{(2)}
=\frac{t_2^{-\alpha}}{(2-\alpha)(1-\alpha)}
  [2(1-\alpha)-\alpha\rho_2],\\
m_{2,2}&=c_1^{(2)}+c_2^{(2)}
=\frac{t_2^{-\alpha}}{(2-\alpha)(1-\alpha)}
  \frac{\alpha+2\rho_2}{\rho_2}.
\end{aligned}
\end{equation}
Then \eqref{matM} and
\eqref{eq:startup-original-coefficients} give
\[
\frac{m_{1,1}}{m_{2,2}}
=\frac{(2-\alpha)\rho_2(1+\rho_2)^\alpha}{\alpha+2\rho_2}.
\]
For fixed $\rho_2\in[0.86,1.16]$, the logarithmic derivative satisfies
\[
\begin{aligned}
\frac{\mathrm d}{\mathrm d\alpha}
\log\left(\frac{m_{1,1}}{m_{2,2}}\right)
&=\log(1+\rho_2)-\frac{1}{2-\alpha}
-\frac{1}{\alpha+2\rho_2}\\
&\leq\log(1+\rho_2)-\frac{2}{1+\rho_2}
\leq\log(2.16)-\frac{2}{2.16}<0.
\end{aligned}
\]
Since the ratio tends to $1$ as $\alpha\to0^+$, we obtain
\begin{equation}
\label{eq:startup-diagonal-comparison}
\frac{m_{1,1}}{m_{2,2}}<1,\qquad 0<\alpha<1.
\end{equation}
Since $0<\eta_2<1$ by Lemma~\ref{lm:eta_prop}, we prove \eqref{eq:3.23_new} also for $k=2$.

Equipped with these structural properties, we now establish the global
$L^2$-norm convergence of the scheme.

\section{$L^2$-norm convergence}\label{sec_4}
We then provide some lower bounds of the recombined L2 coefficients.
\begin{lemma}\label{mkk&mk1}
Assume that the time mesh $0 = t_0 < t_1 < \cdots < t_N =T$ satisfies $\rho_k \in [0.86, 1.16],~k\geq2$, the following properties for the recombined L2 coefficients hold:
\begin{itemize}

\item [(i)] $\frac{1}{\Gamma(1-\alpha)} \bar m_{2,1} \geq  \frac{1}{2\Gamma(2-\alpha)} 
t_2^{-\alpha} $;

\item [(ii)] $\bar m_{k,1}\geq m_{k,1}\geq t_k^{-\alpha}$ for all $k \geq 3$;

\item [(iii)] $\frac{1}{\Gamma(1-\alpha)}m_{k,k} \geq \frac{1}{\Gamma(2-\alpha)}\tau_k^{-\alpha}$ for all $k \geq 2$.
\end{itemize}
\end{lemma}
\begin{proof}
(i) By \eqref{eq:startup-original-coefficients} and \eqref{eta},
\[
\frac{\bar m_{2,1}}{\Gamma(1-\alpha)}
=\frac{1}{\Gamma(2-\alpha)t_2^\alpha}
\left[1+\frac{\alpha}{2(2-\alpha)}
\left(\rho_2^{-1}-\rho_2\right)\right]
\ge\frac{t_2^{-\alpha}}{2\Gamma(2-\alpha)},
\]
since $|\rho_2^{-1}-\rho_2|<1$ and $0<\alpha/(2-\alpha)<1$.


(ii) For $k\ge3$, the integral formula \eqref{eq:aj_cj} gives
\begin{equation}\label{eq:mk1_expansion}
\begin{aligned}
m_{k,1}
&=\int_0^1\frac{2\tau_1(1-\theta)+\tau_2}
{(\tau_1+\tau_2)(t_k-\theta\tau_1)^\alpha}\,\mathrm d\theta\ge t_k^{-\alpha}
\int_0^1\frac{2\tau_1(1-\theta)+\tau_2}{\tau_1+\tau_2}\,
\mathrm d\theta
=t_k^{-\alpha}.
\end{aligned}
\end{equation}
Theorem~\ref{lem:cross_row_recombination}(i) and
\eqref{def_bar_m} then yield
$\bar m_{k,1}=m_{k,1}+\eta_k\bar m_{k,2}\ge m_{k,1}$.

(iii) Since $c_{k-1}^{(k)}>0$, \eqref{eq:a_k1} implies
\[
m_{k,k}=c_{k-1}^{(k)}+c_k^{(k)}
\ge c_k^{(k)}\ge\frac{\tau_k^{-\alpha}}{1-\alpha},
\]
which proves the assertion after division by $\Gamma(1-\alpha)$.

\end{proof}

To establish the optimal $L^2$-norm error estimate, we state the following upper bound for the truncation error $r^k$ defined in \eqref{eq:localrk}.

\begin{lemma}\label{lem:truncation}
Assume that $u\in C^3((0,T];H_0^1(\Omega))$ is the solution to subdiffusion equation \eqref{model_subdiffusion} and satisfies $\|\partial_t^{l} u(t)\|_0\le C_l(1+t^{\alpha-l})$ for $l=1,3$, $0< t\le T$. Assume that the nonuniform mesh satisfies
$\rho_k\in[\rho_{\min},\rho_{\max}]$ for $2\le k\le N$,
where $0<\rho_{\min}<\rho_{\max}<\infty$.
Then the truncation error $r^k$ defined in
\eqref{eq:localrk} satisfies $\|r^1\|_0\le C_r$ and,
for $2\le k\le N$,
\begin{equation}\label{eq:approxerror}
\begin{aligned}
\|r^k\|_0
&\le C_r\Biggl[
\tau_1^\alpha(t_k-t_1)^{-\alpha}
+\sum_{j=2}^{k-1}
\tau_j^4(1+t_{j-1}^{\alpha-3})(t_k-t_j)^{-\alpha-1}
+\tau_k^{3-\alpha}(1+t_{k-1}^{\alpha-3})
\Biggr],
\end{aligned}
\end{equation}
where $C_r>0$ depends only on $T$, $\alpha$,
$\rho_{\min}$, $\rho_{\max}$, $C_1$, and $C_3$,
and is independent of $k$, $N$, and the time steps.
\end{lemma}

\begin{proof}
For any $v\in L^2(\Omega)$ with $\|v\|_0=1$, we have
\begin{equation}
\label{derivative_bounds}
\left|\frac{\mathrm{d}^l}{\mathrm{d}t^l}(u(t),v)\right|
=|(\partial_t^l u(t),v)|
\le \|\partial_t^l u(t)\|_0
\le C_l(1+t^{\alpha-l}),
\qquad l=1,3.
\end{equation}
By the definition of Caputo derivative and the L2 approximation, we have
\[
\begin{aligned}
\left.\partial_t^\alpha(u(t),v)\right|_{t=t_k}
&=(\partial_t^\alpha u(t_k),v), \quad
L_k^\alpha\bigl[(u(\cdot),v)\bigr]
=(L_k^\alpha u,v).
\end{aligned}
\]
Thus the truncation error for $t\mapsto(u(t),v)$ is $(r^k,v)$.
Applying \cite[Lemma~5.1]{quan2022h} gives the stated bounds
with $\|r^k\|_0$ replaced by $|(r^k,v)|$.
The constant $C_r$ is independent of $v$, since
\eqref{derivative_bounds} holds with the same $C_1$ and $C_3$
for every $\|v\|_0=1$.
Taking the supremum and using
$\|r^k\|_0
=\sup\limits_{\substack{v\in L^2(\Omega)\\ \|v\|_0=1}}
|(r^k,v)|$
completes the proof.
\end{proof}


We use Lemma~\ref{lem:truncation} with
 $\rho_{\min}=0.86$ and $\rho_{\max}=1.16$. Multiplying the first-step error equation
$e^1-\Gamma(2-\alpha)\tau_1^\alpha\Delta e^1
=-\Gamma(2-\alpha)\tau_1^\alpha r^1$
by $2e^1$ in $L^2(\Omega)$ and applying Young's inequality gives
\begin{equation}\label{error_1}
\left\|e^1\right\|_0^2 + 2 \Gamma(2-\alpha) \tau_1^{\alpha} \left\|\nabla e^1\right\|_0^2 \leq \left(\Gamma(2-\alpha) \tau_1^{\alpha}\right)^2 \left\|r^1\right\|_0^2.
\end{equation}

With the above lemmas in place, we are now ready to state and prove the main optimal convergence theorem for the L2 method.
\begin{theorem}
[$L^2$-norm error estimation of the L2 method on nonuniform meshes]
\label{mainresult}
Assume that $u\in C^3((0,T];H_0^1(\Omega))$ is the solution to subdiffusion equation \eqref{model_subdiffusion} and satisfies $\|\partial_t^{l} u(t)\|_0\le C_l(1+t^{\alpha-l})$ for $l=1,3$, $0< t\le T$.
If the nonuniform mesh satisfies $\rho_k\in [0.86, 1.16]$ for $k\geq2$, then the numerical solutions $u^k$ of L2 scheme \eqref{eq:sch_sub_graded} have the following error estimate:
\begin{equation}\label{eq:main_error_bound}
\begin{aligned}
\|e^k\|_0 + \sqrt{m_{k,k}^{-1}}\|\nabla e^k\|_0 &\leq C_\eta\Bigg[ C_r^2\Gamma^2(2-\alpha)\tau_1^{2\alpha}
+C\max_{2 \leq i \leq k}\Bigg( t_i^\alpha\bigg( \tau_1^\alpha(t_i-t_1)^{-\alpha} \\
&\qquad {}+ \sum_{j=2}^{i-1}\tau_j^4(1+t_{j-1}^{\alpha-3})(t_i-t_j)^{-\alpha-1}
+ \tau_i^{3-\alpha}(1+t_{i-1}^{\alpha-3}) \bigg)\Bigg)^2 \, \Bigg]^{\frac{1}{2}}
\\
&\leq\bar C\max_{1\leq j\leq k}
\tau_j^{3-\alpha}t_j^{2\alpha-3}
\qquad 2\leq k\leq N.
\end{aligned}
\end{equation}
Here
\[
C_\eta
\coloneqq
1+\frac{1}{1-\eta_*}
=
\frac{25+8\sqrt{2}}{7}
\approx 5.1877,
\qquad
\eta_*
\coloneqq
\frac{3+2\sqrt{2}}{2+4\sqrt{2}}.
\]
\(C_r\) is defined in Lemma~\ref{lem:truncation}, and
\(C>0\) may depend on \(T\), \(\alpha\), and the regularity
constants \(C_l\), \(l=1,3\), but is independent of \(k\), \(N\), and the time
steps.
The constant $\bar C>0$ depends only on $C_\eta$, $C_r$, $C$,
$T$, and $\alpha$,
but is independent of $k$, $N$, and the time steps.
\end{theorem}

\begin{proof}
Throughout the proof, \(C>0\) denotes a generic constant with the
dependence specified in the theorem.

Fix an arbitrary time level \(k\geq3\). Throughout the following
induction, \(k\) and \(\eta_k\) are fixed, and we write
\[
\bar e^j\coloneqq e^j-\eta_k e^{j-1},
\qquad 1\leq j\leq k,
\qquad
\bar e^0\coloneqq e^0=0.
\]
For every \(2\leq n\leq k\), the error equation \eqref{error_eq} and
the definition \eqref{eq:cross_row_coefficients} give
\begin{equation}\label{eq:common_row_error}
\begin{aligned}
\bar e^n-\Gamma(1-\alpha)m_{n,n}^{-1}\Delta e^n
={}&
m_{n,n}^{-1}\sum_{j=1}^{n-1}
\big(\bar m_{n,j+1}^{(k)}-\bar m_{n,j}^{(k)}\big)\bar e^j
-\Gamma(1-\alpha)m_{n,n}^{-1}r^n.
\end{aligned}
\end{equation}
Multiplying both sides of \eqref{eq:common_row_error} by
\(2\bar e^n\) in $L^2 (\Omega)$, and using
\(
2(\nabla e^n,\nabla\bar e^n)
=
\|\nabla e^n\|_0^2+\|\nabla\bar e^n\|_0^2
-\eta_k^2\|\nabla e^{n-1}\|_0^2
\),
we obtain
\begin{align}
&\|\bar e^n\|_0^2
+\Gamma(1-\alpha)m_{n,n}^{-1}\|\nabla e^n\|_0^2
+\frac{\bar m_{n,1}^{(k)}}{m_{n,n}}\|\bar e^n\|_0^2
\notag\\
&\quad\leq
\sum_{j=1}^{n-1}
\frac{\bar m_{n,j+1}^{(k)}-\bar m_{n,j}^{(k)}}{m_{n,n}}
\|\bar e^j\|_0^2
+\Gamma(1-\alpha)m_{n,n}^{-1}\eta_k^2
\|\nabla e^{n-1}\|_0^2
-2\Gamma(1-\alpha)m_{n,n}^{-1}(r^n,\bar e^n).
\label{eq:common_energy_before_transfer}
\end{align}
Here the nonnegative term
\(\Gamma(1-\alpha)m_{n,n}^{-1}\|\nabla\bar e^n\|_0^2\)
has been dropped. By Theorem~\ref{lem:cross_row_recombination}(iv), the
gradient term carried from level \(n-1\) satisfies
\[
\begin{aligned}
&\Gamma(1-\alpha)m_{n,n}^{-1}\eta_k^2
\|\nabla e^{n-1}\|_0^2
\leq
\frac{\bar m_{n,n}^{(k)}-\bar m_{n,n-1}^{(k)}}{m_{n,n}}
\Gamma(1-\alpha)m_{n-1,n-1}^{-1}
\|\nabla e^{n-1}\|_0^2.
\end{aligned}
\]
Adding, if \(n\geq3\), the nonnegative gradient contributions
corresponding to \(1\leq j\leq n-2\) to the right-hand side of
\eqref{eq:common_energy_before_transfer}, we have
\begin{align}
&\|\bar e^n\|_0^2
+\Gamma(1-\alpha)m_{n,n}^{-1}\|\nabla e^n\|_0^2
+\frac{\bar m_{n,1}^{(k)}}{m_{n,n}}\|\bar e^n\|_0^2
\notag\\
&\quad\leq
\sum_{j=1}^{n-1}
\frac{\bar m_{n,j+1}^{(k)}-\bar m_{n,j}^{(k)}}{m_{n,n}}
\left(
\|\bar e^j\|_0^2
+\Gamma(1-\alpha)m_{j,j}^{-1}\|\nabla e^j\|_0^2
\right)
-2\Gamma(1-\alpha)m_{n,n}^{-1}(r^n,\bar e^n).
\label{eq:common_energy_closed}
\end{align}
The last term is estimated by
\[
\begin{aligned}
-2\Gamma(1-\alpha)m_{n,n}^{-1}(r^n,\bar e^n)
\leq{}&
\frac{\bar m_{n,1}^{(k)}}{m_{n,n}}\|\bar e^n\|_0^2
+
\frac{\bar m_{n,1}^{(k)}}{m_{n,n}}
\left(\frac{\Gamma(1-\alpha)}{\bar m_{n,1}^{(k)}}\right)^2
\|r^n\|_0^2.
\end{aligned}
\]
Consequently,
\begin{equation}\label{eq:pre_induction}
\begin{aligned}
&\|\bar e^n\|_0^2
+\Gamma(1-\alpha)m_{n,n}^{-1}\|\nabla e^n\|_0^2
\\
&\leq
\sum_{j=1}^{n-1}
\frac{\bar m_{n,j+1}^{(k)}-\bar m_{n,j}^{(k)}}{m_{n,n}}
\left(
\|\bar e^j\|_0^2
+\Gamma(1-\alpha)m_{j,j}^{-1}\|\nabla e^j\|_0^2
\right)
+
\frac{\bar m_{n,1}^{(k)}}{m_{n,n}}
\left(\frac{\Gamma(1-\alpha)}{\bar m_{n,1}^{(k)}}\right)^2
\|r^n\|_0^2.
\end{aligned}
\end{equation}
We next estimate the coefficients
\(\Gamma(1-\alpha)/\bar m_{i,1}^{(k)}\)
for \(2\leq i\leq k\). For \(i=2\), \eqref{App-eq:cross-L1} and Lemma~\ref{mkk&mk1}(iii) give
\begin{equation}\label{last_term_bound1}
\frac{\Gamma(1-\alpha)}{\bar m_{2,1}^{(k)}}
\leq
5\frac{\Gamma(1-\alpha)}{m_{2,2}}
\leq
5\Gamma(2-\alpha)\tau_2^\alpha
\leq
5\Gamma(2-\alpha)t_2^\alpha.
\end{equation}
For \(3\leq i\leq k\),
Theorem~\ref{lem:cross_row_recombination}(i), the backward recursion
\eqref{eq:cross_row_backward}, and 
Lemma~\ref{mkk&mk1}(ii) give
\begin{equation}\label{last_term_bound2}
\bar m_{i,1}^{(k)}
=m_{i,1}+\eta_k\bar m_{i,2}^{(k)}
\geq m_{i,1}\geq t_i^{-\alpha}.
\end{equation}
Combining \eqref{last_term_bound1} and \eqref{last_term_bound2}, it follows that
\begin{equation}\label{eq:cross_first_coefficient_bound}
\max_{2\leq i\leq k}
\left(\frac{\Gamma(1-\alpha)}{\bar m_{i,1}^{(k)}}\right)^2
\|r^i\|_0^2
\leq
C\max_{2\leq i\leq k}t_i^{2\alpha}\|r^i\|_0^2.
\end{equation}
We now use induction with respect to \(n\) to prove
\begin{equation}\label{induction_ineq}
\begin{aligned}
&\|\bar e^n\|_0^2
+\Gamma(1-\alpha)m_{n,n}^{-1}\|\nabla e^n\|_0^2\\
&\quad\leq
\|e^1\|_0^2
+\Gamma(1-\alpha)m_{1,1}^{-1}\|\nabla e^1\|_0^2
+C\max_{2\leq i\leq k}t_i^{2\alpha}\|r^i\|_0^2,
\qquad 1\leq n\leq k.
\end{aligned}
\end{equation}
The case \(n=1\) follows from \(\bar e^1=e^1\).
For \(2\leq n\leq k\), suppose that
\eqref{induction_ineq} holds for all \(1\leq j\leq n-1\).
The weights in the history sum are nonnegative by
Theorem~\ref{lem:cross_row_recombination}(ii).
Using the induction hypothesis and
Theorem~\ref{lem:cross_row_recombination}(iii), we obtain
\begin{equation}\label{eq:induction_history_bound}
\begin{aligned}
&\sum_{j=1}^{n-1}
\frac{\bar m_{n,j+1}^{(k)}-\bar m_{n,j}^{(k)}}{m_{n,n}}
\left(
\|\bar e^j\|_0^2
+\Gamma(1-\alpha)m_{j,j}^{-1}\|\nabla e^j\|_0^2
\right)\\
&\quad\leq
\left(
1-\frac{\bar m_{n,1}^{(k)}}{m_{n,n}}
\right)
\bigg(
\|e^1\|_0^2
+\Gamma(1-\alpha)m_{1,1}^{-1}\|\nabla e^1\|_0^2
+C\max_{2\leq i\leq k}t_i^{2\alpha}\|r^i\|_0^2
\bigg).
\end{aligned}
\end{equation}
Moreover, since
\(\bar m_{n,1}^{(k)}/m_{n,n}>0\),
\eqref{eq:cross_first_coefficient_bound} gives
\begin{equation}\label{eq:induction_residual_bound}
\begin{aligned}
\frac{\bar m_{n,1}^{(k)}}{m_{n,n}}
\left(
\frac{\Gamma(1-\alpha)}{\bar m_{n,1}^{(k)}}
\right)^2
\|r^n\|_0^2
&\leq
\frac{\bar m_{n,1}^{(k)}}{m_{n,n}}
\bigg(
\|e^1\|_0^2
+\Gamma(1-\alpha)m_{1,1}^{-1}\|\nabla e^1\|_0^2
+C\max_{2\leq i\leq k}t_i^{2\alpha}\|r^i\|_0^2
\bigg).
\end{aligned}
\end{equation}
Substituting \eqref{eq:induction_history_bound} and
\eqref{eq:induction_residual_bound} into
\eqref{eq:pre_induction},
we obtain \eqref{induction_ineq} at level \(n\).
This completes the induction.
Taking \(n=k\), and recalling that
\(\bar m_{k,j}^{(k)}=\bar m_{k,j}\), gives
\begin{equation}\label{eq:target_row_estimate}
\begin{aligned}
&\|e^k-\eta_k e^{k-1}\|_0^2
+\Gamma(1-\alpha)m_{k,k}^{-1}\|\nabla e^k\|_0^2\\
&\quad\leq
\|e^1\|_0^2
+\Gamma(1-\alpha)m_{1,1}^{-1}\|\nabla e^1\|_0^2
+C\max_{2\leq i\leq k}t_i^{2\alpha}\|r^i\|_0^2.
\end{aligned}
\end{equation}

For \(k=2\), the same two-level calculation, together with
\(\bar m_{2,1}=(1-\eta_2)m_{2,2}\) and
\eqref{eq:3.23_new}, gives
\[
\begin{aligned}
&\|e^2-\eta_2e^1\|_0^2
+\Gamma(1-\alpha)m_{2,2}^{-1}\|\nabla e^2\|_0^2\\
&\quad\leq
\eta_2\left(
\|e^1\|_0^2
+\Gamma(1-\alpha)m_{1,1}^{-1}\|\nabla e^1\|_0^2
\right)
+(1-\eta_2)
\left(
\frac{\Gamma(1-\alpha)}{\bar m_{2,1}}
\right)^2
\|r^2\|_0^2\\
&\quad\leq
\|e^1\|_0^2
+\Gamma(1-\alpha)m_{1,1}^{-1}\|\nabla e^1\|_0^2
+C t_2^{2\alpha}\|r^2\|_0^2,
\end{aligned}
\]
where the last inequality follows from \(0<\eta_2<1\)
and Lemma~\ref{mkk&mk1}(i).
Hence \eqref{eq:target_row_estimate} holds for all
\(2\leq k\leq N\).

Since \eqref{eq:target_row_estimate} holds at every target level
with the same constant $C$, discarding its gradient term gives
\begin{equation}\label{eq:row_difference_bound}
\|e^j-\eta_j e^{j-1}\|_0
\le
\bigg(
\|e^1\|_0^2
+\Gamma(1-\alpha)m_{1,1}^{-1}\|\nabla e^1\|_0^2
+C\max_{2\le i\le k}t_i^{2\alpha}\|r^i\|_0^2
\bigg)^{1/2},
\qquad 2\le j\le k.
\end{equation}
The same right-hand side also bounds $\|e^1\|_0$.
Successive use of
$\|e^j\|_0\le\|e^j-\eta_j e^{j-1}\|_0+\eta_*\|e^{j-1}\|_0$
and $\sum_{\ell=0}^{k-1}\eta_*^\ell\le(1-\eta_*)^{-1}$ yields
\begin{equation}\label{eq:error_reconstruction_bound}
\|e^k\|_0
\le\frac{1}{1-\eta_*}
\bigg(
\|e^1\|_0^2
+\Gamma(1-\alpha)m_{1,1}^{-1}\|\nabla e^1\|_0^2
+C\max_{2\le i\le k}t_i^{2\alpha}\|r^i\|_0^2
\bigg)^{1/2}.
\end{equation}
Retaining the gradient term in \eqref{eq:target_row_estimate}
and using $\Gamma(1-\alpha)\ge1$ gives
\begin{equation}\label{eq:gradient_reconstruction_bound}
\sqrt{m_{k,k}^{-1}}\|\nabla e^k\|_0
\le
\bigg(
\|e^1\|_0^2
+\Gamma(1-\alpha)m_{1,1}^{-1}\|\nabla e^1\|_0^2
+C\max_{2\le i\le k}t_i^{2\alpha}\|r^i\|_0^2
\bigg)^{1/2}.
\end{equation}

Adding \eqref{eq:error_reconstruction_bound} and
\eqref{eq:gradient_reconstruction_bound}, we obtain
\begin{equation}\label{result}
\begin{aligned}
&\|e^k\|_0+\sqrt{m_{k,k}^{-1}}\|\nabla e^k\|_0\leq
C_\eta
\bigg(
\|e^1\|_0^2
+\Gamma(1-\alpha)m_{1,1}^{-1}\|\nabla e^1\|_0^2
+C\max_{2\leq i\leq k}t_i^{2\alpha}\|r^i\|_0^2
\bigg)^{1/2},
\end{aligned}
\end{equation}
with \(C_\eta=1+(1-\eta_*)^{-1}\).

Since $\Gamma(1-\alpha)m_{1,1}^{-1}
=\Gamma(2-\alpha)\tau_1^\alpha$,
the initial-step estimate \eqref{error_1} yields

\begin{equation}\label{eq:initial_energy_bound}
\|e^1\|_0^2
+\Gamma(1-\alpha)m_{1,1}^{-1}\|\nabla e^1\|_0^2
\leq
\Gamma^2(2-\alpha)\tau_1^{2\alpha}\|r^1\|_0^2.
\end{equation}
Finally, substituting \eqref{eq:initial_energy_bound}
and the truncation error bounds from
Lemma~\ref{lem:truncation} into \eqref{result}
gives the first inequality in \eqref{eq:main_error_bound}.
Elementary estimates of the right-hand side of the first inequality
yield the second inequality in \eqref{eq:main_error_bound}.
\end{proof}

Next, we specialize our analysis on the modified graded mesh proposed by Kopteva in \cite{kopteva2021error} to establish the optimal \(L^2\)-norm convergence of the L2 method.
For a grading parameter \(r>1\), define
\begin{equation}\label{eq:k0}
    k_0 \coloneqq \min \left\{ k \in \mathbb{N} : \frac{(k+2)^r - (k+1)^r}{(k+1)^r - k^r} \leq \rho^* \right\}.
\end{equation}
The modified graded mesh is then defined by
\begin{equation}\label{eq:mod_mesh}
    t_k = T\frac{(k+k_0)^r - k_0^r}{(N+k_0)^r - k_0^r}, \quad \tau_k = t_k - t_{k-1}, \quad 1 \leq k \leq N.
\end{equation}
Note that when $k_0 = 0$, this modified graded mesh \eqref{eq:mod_mesh} coincides with the standard graded mesh. It is not difficult to verify that there exists a constant \(C_{\mathrm{mg}}>0\), depending only on
\(k_0\) and \(r\), such that the modified graded mesh satisfies
\begin{equation}\label{eq:mod_mesh_prop}
t_k
\leq C_{\mathrm{mg}}T N^{-r}k^r,
\qquad
\tau_k
\leq C_{\mathrm{mg}}T N^{-r}k^{r-1},
\qquad
1\leq\frac{\tau_k}{\tau_{k-1}}\leq\rho^*,
\qquad 2 \leq k \leq N.
\end{equation}

\begin{lemma}\label{lem:truncation_mod}
    For the modified graded mesh \eqref{eq:mod_mesh}, the truncation error $r^k$ satisfies $\|r^k\|_0 \leq C_r k^{-\min\{r\alpha, 3-\alpha\}}$.
Here the regularity assumptions of Lemma~\ref{lem:truncation} are imposed,
and $C_r>0$ may depend on $T$, $\alpha$, $C_l$ for $l=1,3$, the fixed parameters $r$, $k_0$, and the
fixed bounds $\rho_{\min}$, $\rho_{\max}$, but is independent of $k$,
$N$ and the time steps.
\end{lemma}

\begin{proof}
The stated $L^2$ truncation error estimate follows from
\cite[Lemma~5.3]{quan2022h} by the same duality argument
as in the proof of Lemma~\ref{lem:truncation}.
\end{proof}

\begin{theorem}[Optimal error estimate on modified graded mesh]
\label{thm:optimal_error}
Assume that $u\in C^3((0,T];H_0^1(\Omega))$ is the solution to subdiffusion equation \eqref{model_subdiffusion} and satisfies $\|\partial_t^{l} u(t)\|_0\le C_l(1+t^{\alpha-l})$ for $l=1,3$, $0< t\le T$.
If the modified graded mesh \eqref{eq:mod_mesh}, with grading parameter
\(r>1\) and the fixed value \(\rho^\ast=1.16\) in
\eqref{eq:k0}, is employed, then the numerical solutions \(u^n\) of
the L2 scheme \eqref{eq:sch_sub_graded} satisfy, for \(1\leq n\leq N\),
\begin{equation}\label{eq:err_final_asymptotic}
\|e^n\|_0
+\sqrt{m_{n,n}^{-1}}\|\nabla e^n\|_0
\leq
\begin{cases}
\widetilde C N^{-r\alpha},
&
\displaystyle
1<r\leq\frac{3-\alpha}{\alpha},
\\[2mm]
\widetilde C N^{-(3-\alpha)},
&
\displaystyle
r>\frac{3-\alpha}{\alpha},
\end{cases}
\end{equation}
where \(\widetilde C>0\) may depend on \(T\), \(\alpha\), \(\Omega\),
and the regularity constants \(C_l\), \(l=1,3\), \(r\), \(k_0\), but is independent of
\(N\) and \(n\).
\end{theorem}

\begin{proof}

Throughout the proof, $\widetilde C>0$ denotes a generic constant
with the dependence stated in the theorem. Set
$p\coloneqq\min\{r\alpha,3-\alpha\}$.
By \eqref{error_1}, $m_{1,1}^{-1}=(1-\alpha)\tau_1^\alpha$,
Lemma~\ref{lem:truncation}, and \eqref{eq:mod_mesh_prop},
\[
\|e^1\|_0+\sqrt{m_{1,1}^{-1}}\|\nabla e^1\|_0
\le\widetilde C\tau_1^\alpha
\le\widetilde C N^{-r\alpha}
\le\widetilde C N^{-p}.
\]
For $2\le n\le N$, combining \eqref{result},
\eqref{eq:initial_energy_bound}, and Lemma~\ref{lem:truncation_mod}
with the node bounds in \eqref{eq:mod_mesh_prop} gives
\[
\begin{aligned}
\|e^n\|_0+\sqrt{m_{n,n}^{-1}}\|\nabla e^n\|_0
&\le C_\eta\left[
C_r^2\Gamma^2(2-\alpha)\tau_1^{2\alpha}
+CC_r^2\max_{2\le i\le n}t_i^{2\alpha}i^{-2p}
\right]^{1/2}\\
&\le\widetilde C N^{-r\alpha}
\left(1+\max_{2\le i\le n}i^{2(r\alpha-p)}\right)^{1/2}
\le\widetilde C N^{-p}.
\end{aligned}
\]
The last inequality uses $r\alpha-p\ge0$ and $n\le N$.
Together with the first-step estimate, this proves
\eqref{eq:err_final_asymptotic}.
\end{proof}

\section{Numerical experiments}\label{sec:numerical}

We first examine the error estimate on random meshes and the optimal
convergence rates on modified graded meshes using the same exact
solution. We then compare the accuracy and computational cost of
different time meshes for a solution with varying temporal scales.
In all computations, the L2 coefficients are evaluated by the
threshold-condition and Taylor-expansion technique in
\cite[Algorithm~3.2]{quan2024roundoff}, with $\delta=10^{-11}$,
to avoid the round-off errors. 
We employ the spectral collocation method
\cite{trefethen2000spectral,shen2011spectral} in space with
$20^2$ Chebyshev--Gauss--Lobatto points in all three experiments.

\subsection{Error estimates on random and modified graded meshes}
\label{sec:num_power_test}

Consider the subdiffusion equation \eqref{model_subdiffusion} with
$T=1$, $\Omega=[-1,1]^2$, homogeneous Dirichlet boundary conditions,
and $u^0=0$. We use the same weakly singular solution as in the numerical experiment of \cite{quan2022h}, namely,
\begin{equation}\label{eq:num_exact_forcing}
\begin{aligned}
u(t,x,y)&=t^\alpha\sin(\pi x)\sin(\pi y),\\
f(t,x,y)&=\bigl[\Gamma(1+\alpha)+2\pi^2t^\alpha\bigr]
             \sin(\pi x)\sin(\pi y).
\end{aligned}
\end{equation}
The time derivative behaves like $t^{\alpha-1}$ near $t=0$, so that this example exhibits the initial weak singularity allowed in our analysis.

We define the global \(L^2\) error by
\begin{equation}\label{eq:num_max_error}
E_N^{\max}
\coloneqq
\max_{1\leq k\leq N}\|u(t_k)-u^k\|_0.
\end{equation}

\subsubsection{Random nonuniform meshes}\label{sec:num_nonuniform}
We first examine \eqref{eq:main_error_bound} for
$\alpha=0.3,0.5,0.7$. For a prescribed initial step, we generate
independent random ratios $\rho_k\sim\mathcal U[0.86,1.16]$, where $\mathcal U$ denotes the uniform distribution.
We take $\tau_1=10^{-6},10^{-9},10^{-12}$, with random seeds
$17,43,97$ in MATLAB R2024b, respectively. These initial steps allow us to examine
the estimate over several scales of the initial mesh resolution.
The same three meshes are used for each value of $\alpha$.
The final few steps are adjusted to reach $T=1$ while preserving
$0.86\leq\rho_k\leq1.16$. The resulting numbers of steps are
$N=1843,2067,2715$, respectively.

Let $L_k$ and $R_k$ denote the left- and right-hand sides of the first
inequality in \eqref{eq:main_error_bound}, respectively.
For each $\alpha$, the constants in $R_k$ are fixed a priori,
independently of the computed errors.
Tracing the constants in \cite[Lemma~5.1]{quan2022h} for
\eqref{eq:num_exact_forcing} gives
$C_r\approx0.3218947870$, $0.2768465166$, and $0.2377764221$
for $\alpha=0.3,0.5,0.7$, respectively.
A sufficient choice of $C$ in \eqref{eq:main_error_bound} is
\begin{equation}\label{eq:num_constants_used}
C=C_r^2\max\bigl\{\Gamma^2(1-\alpha),\,
25\Gamma^2(2-\alpha)\bigr\}.
\end{equation}
The constant $C_\eta$ is kept as defined in Theorem~\ref{mainresult}.

Figure~\ref{fig:num_random_ratio} shows $L_k/R_k<1$ in all nine runs,
consistent with \eqref{eq:main_error_bound} on these nonmonotone meshes.
The maxima over the three meshes are approximately
$3.10049\times10^{-3}$, $6.04931\times10^{-3}$, and
$7.22319\times10^{-3}$ for $\alpha=0.3,0.5,0.7$, respectively.

\begin{figure}[tbp]
\centering
\includegraphics[width=\textwidth]{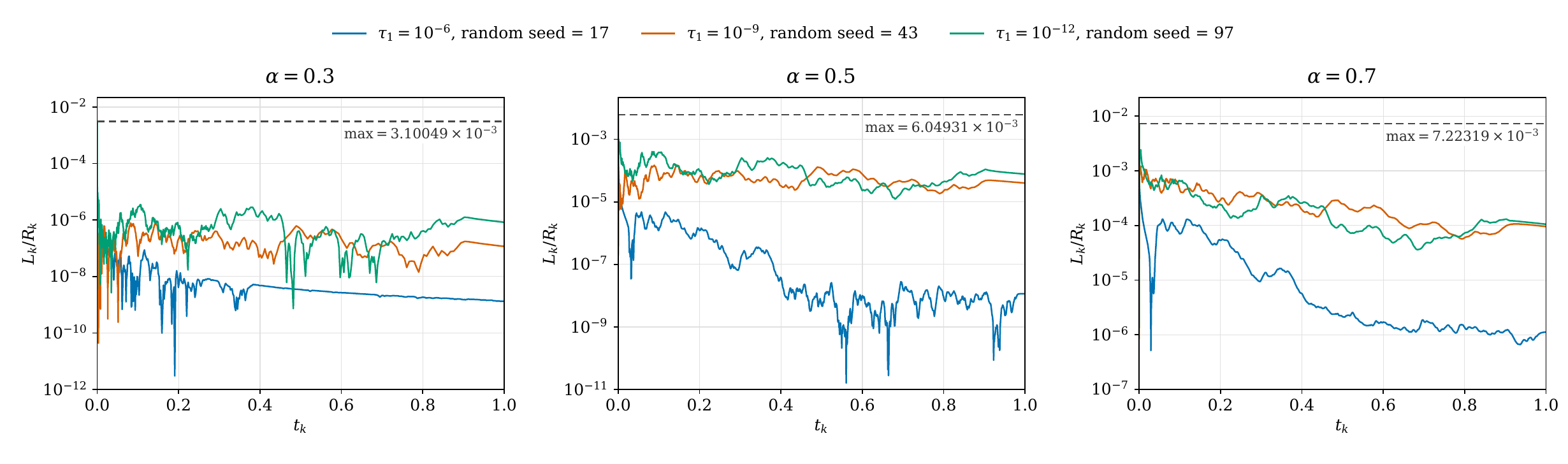}
\caption{The ratios $L_k/R_k$ on random nonuniform meshes.
From left to right, $\alpha=0.3,0.5,0.7$. Each curve corresponds to
the initial step and random seed given in the legend. The horizontal
axes show physical time on a linear scale; the vertical axes are
logarithmic and have different ranges. Each dashed horizontal line
marks the observed maximum over the three meshes in that panel.}
\label{fig:num_random_ratio}
\end{figure}

\subsubsection{Modified graded meshes}\label{sec:num_graded}
We now examine the convergence on the modified graded meshes
\eqref{eq:mod_mesh}. In particular, we consider grading parameters below,
at, and above the threshold $(3-\alpha)/\alpha$ in
Theorem~\ref{thm:optimal_error}.

We take $\alpha=0.3,\,0.5,\,0.7$ and
\[
r=\frac{2}{\alpha},\qquad
r=\frac{3-\alpha}{\alpha},\qquad
r=\frac{5-\alpha}{\alpha},
\]
with $N=8000,\,10000,\,12000,\,14000,\,16000$.
The observed order between two
successive meshes is computed as
\begin{equation}\label{eq:num_order}
\mathrm{Order}_j=
\frac{\log\bigl(E_{N_{j-1}}^{\max}/E_{N_j}^{\max}\bigr)}
     {\log(N_j/N_{j-1})},
\qquad j=2,\ldots,5.
\end{equation}

Table~\ref{tab:num_max_errors} shows orders close to the rates
$2$, $3-\alpha$, and $3-\alpha$ predicted by
Theorem~\ref{thm:optimal_error} for the three choices of $r$.
For $\alpha=0.3$ and $r=(5-\alpha)/\alpha$, the larger order on
the first refinement interval is followed by orders near $3-\alpha$. These results are consistent with the predicted rate.

\begin{table}[tbp]
\centering
\captionsetup{position=top}
\caption{Global $L^2$ errors $E_N^{\max}$ and convergence
orders for the L2 scheme on modified graded meshes with different
$\alpha$, $r$, and $N$. Each error row is followed by the observed
orders defined in \eqref{eq:num_order}.}
\label{tab:num_max_errors}

\begingroup
\small
\setlength{\tabcolsep}{3.5pt}
\renewcommand{\arraystretch}{1.10}
\begin{tabular*}{\linewidth}{@{\extracolsep{\fill}}ccc*{5}{c}@{}}
\toprule
$\alpha$ & $r$ & $k_0$
& $N=8000$ & $N=10000$ & $N=12000$ & $N=14000$ & $N=16000$ \\
\midrule
$\alpha=0.3$ & $r=\frac{2}{\alpha}$ & 38
& 2.5015e-06 & 1.6041e-06 & 1.1154e-06 & 8.2027e-07 & 6.2846e-07 \\
& & & --- & 1.9911 & 1.9928 & 1.9940 & 1.9948 \\
& $r=\frac{3-\alpha}{\alpha}$ & 53
& 1.4269e-07 & 7.8393e-08 & 4.8031e-08 & 3.1732e-08 & 2.2155e-08 \\
& & & --- & 2.6840 & 2.6870 & 2.6890 & 2.6905 \\
& $r=\frac{5-\alpha}{\alpha}$ & 98
& 1.0627e-10 & 4.7189e-11 & 2.8653e-11 & 1.8789e-11 & 1.3035e-11 \\
& & & --- & 3.6380 & 2.7363 & 2.7375 & 2.7381 \\
\midrule
$\alpha=0.5$ & $r=\frac{2}{\alpha}$ & 20
& 6.1952e-07 & 3.9689e-07 & 2.7581e-07 & 2.0273e-07 & 1.5527e-07 \\
& & & --- & 1.9954 & 1.9963 & 1.9969 & 1.9973 \\
& $r=\frac{3-\alpha}{\alpha}$ & 26
& 5.8414e-08 & 3.3492e-08 & 2.1255e-08 & 1.4469e-08 & 1.0368e-08 \\
& & & --- & 2.4927 & 2.4941 & 2.4950 & 2.4957 \\
& $r=\frac{5-\alpha}{\alpha}$ & 53
& 2.5430e-10 & 1.4577e-10 & 9.2486e-11 & 6.2940e-11 & 4.5091e-11 \\
& & & --- & 2.4937 & 2.4955 & 2.4967 & 2.4976 \\
\midrule
$\alpha=0.7$ & $r=\frac{2}{\alpha}$ & 12
& 1.5989e-07 & 1.0239e-07 & 7.1135e-08 & 5.2277e-08 & 4.0033e-08 \\
& & & --- & 1.9973 & 1.9978 & 1.9981 & 1.9984 \\
& $r=\frac{3-\alpha}{\alpha}$ & 15
& 3.5760e-08 & 2.1423e-08 & 1.4093e-08 & 9.8904e-09 & 7.2772e-09 \\
& & & --- & 2.2961 & 2.2968 & 2.2973 & 2.2977 \\
& $r=\frac{5-\alpha}{\alpha}$ & 34
& 6.3764e-10 & 3.8226e-10 & 2.5159e-10 & 1.7662e-10 & 1.2998e-10 \\
& & & --- & 2.2930 & 2.2943 & 2.2952 & 2.2959 \\
\bottomrule
\end{tabular*}
\endgroup
\end{table}

\subsection{Comparison of uniform and nonuniform time meshes}
\label{sec:num_adaptive}

We next compare modified graded, adaptive, and fixed-step meshes
after a common startup using the modified graded mesh. Here uniform stepping refers to the fixed
steps after this startup.
Consider \eqref{model_subdiffusion} on $\Omega=(-1,1)^2$, with
$T=10$, homogeneous Dirichlet boundary conditions, zero initial
data, and
\begin{equation}\label{eq:num_adaptive_source}
\begin{aligned}
f(t,x,y)&=q(t)\sin(\pi x)\sin(\pi y),\\
q(t)&\coloneqq1+e^{-4(t-2)^2}+e^{-16(t-5)^2}+e^{-64(t-8)^2}.
\end{aligned}
\end{equation}
The exact solution admits the integral representation
\begin{equation}\label{eq:num_reference_solution}
u(t,x,y)=\sin(\pi x)\sin(\pi y)
\int_0^t (t-s)^{\alpha-1}
E_{\alpha,\alpha}\!\left(-2\pi^2(t-s)^\alpha\right)
q(s)\,\mathrm{d}s,
\end{equation}
where $E_{\alpha,\beta}(z)=\sum_{j=0}^{\infty}
z^j/\Gamma(\alpha j+\beta)$ is the two-parameter
Mittag--Leffler function.
We compute the reference solution using the composite trapezoidal
rule after a logarithmic change of variables, with $1009$, $729$,
and $594$ quadrature nodes for $\alpha=0.3$, $0.5$, and $0.7$,
respectively (doubling each number of quadrature nodes only changes the reference
values by less than $6\times10^{-15}$ at all computed time nodes,
well below the reported numerical errors). This computation is independent of the L2
discretization.
We use this example to compare the accuracy and computational
cost of the three meshes for a solution with an initial weak
singularity and rapid changes at later times. 

We first specify the common graded startup for the three meshes.
For $\alpha=0.3,0.5,0.7$, we construct a modified graded mesh
\eqref{eq:mod_mesh}, denoted by $\{t_j^G\}_{j=0}^{N_G}$, with
\[
r=\frac{3-\alpha}{\alpha}+2,\qquad
(k_0,N_G)=(67,397),\ (40,806),\ (28,1017),
\]
respectively. Here $k_0$ is selected by \eqref{eq:k0}, and $N_G$
replaces $N$ in \eqref{eq:mod_mesh}. 
For each fixed $\alpha$, we set $t_s=0.01$ and define
\begin{equation}\label{eq:num_shared_start}
\tau_{\min}=t_p^G-t_{p-1}^G,\quad
\text{with }p=\max\{j:t_j^G\leq t_s\}.
\end{equation}
Here, $p=180,275,254$ for $\alpha = 0.3,0.5,0.7$, respectively.
In each case, the maximum error on the adaptive mesh occurs
after this startup.

\begin{figure}[tbp]
\centering
\includegraphics[width=\textwidth]{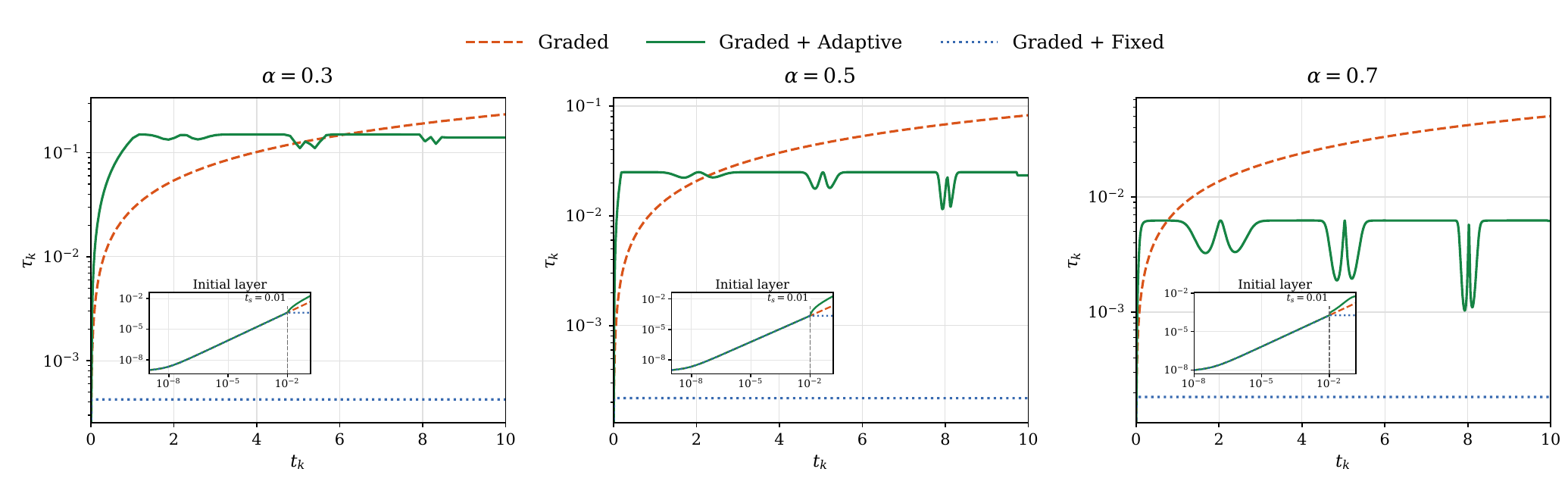}
\caption{Time-step evolution on the three meshes.
From left to right, \(\alpha=0.3,0.5,0.7\).
The insets labelled ``Initial layer'' show the common modified
graded startup used to resolve the initial weak singularity,
together with the subsequent separation of the meshes.
The startup cutoff is \(t_s=0.01\); the meshes separate at the
last graded node not exceeding this value, as defined in
\eqref{eq:num_shared_start}.}
\label{fig:num_adaptive_steps}
\end{figure}

\begin{table}[tbp]
\centering
\captionsetup{position=top}
\caption{Maximum $L^2$ errors, total step counts, and elapsed times
on the three meshes. All computations include the common startup
selected by $t_s=0.01$ and end at $t_N=10$.}
\label{tab:num_adaptive_cost}
\small
\setlength{\tabcolsep}{5pt}
\begin{tabular}{@{}llllll@{}}
\toprule
$\alpha$ & Mesh & $N$ & $t_N$ & $E_N^{\max}$ & Time (s) \\
\midrule
0.3 & Graded & 397 & 10 & $1.064\times10^{-3}$ & 0.19024 \\
 & Graded + Adaptive & 282 & 10 & $5.954\times10^{-4}$ & 0.13287 \\
 & Graded + Fixed & 23713 & 10 & $3.888\times10^{-4}$ & 352.85 \\
\addlinespace
0.5 & Graded & 806 & 10 & $3.580\times10^{-4}$ & 0.66486 \\
 & Graded + Adaptive & 722 & 10 & $1.986\times10^{-5}$ & 0.56172 \\
 & Graded + Fixed & 46019 & 10 & $7.641\times10^{-6}$ & 1348.7 \\
\addlinespace
0.7 & Graded & 1017 & 10 & $3.624\times10^{-4}$ & 1.0135 \\
 & Graded + Adaptive & 2436 & 10 & $1.062\times10^{-6}$ & 4.89 \\
 & Graded + Fixed & 54816 & 10 & $5.103\times10^{-7}$ & 1908.3 \\
\bottomrule
\end{tabular}
\end{table}

After $t_p^G$, the modified graded mesh follows its remaining
prescribed nodes, while the fixed-step mesh uses the constant
target step $\tau_{\min}$.
For the adaptive mesh, we use the roughness-based strategy
in \cite{qiao2011adaptive}. Define
\begin{equation}\label{eq:num_roughness}
w^k=\left(\frac{1}{|\Omega|}
\int_\Omega (u^k-\bar u^k)^2\,\mathrm dx\right)^{1/2}
,\qquad
\omega_k=\frac{w^k-w^{k-1}}{\tau_k},
\end{equation}
where $\bar u^k$ is the spatial mean, which is zero up to
roundoff in this example.
The proposed adaptive step is
\begin{equation}\label{eq:num_continuations}
\widehat\tau_{k+1}
=\max\left\{\tau_{\min},
\frac{\tau_{\max}}{\sqrt{1+\beta\omega_k^2}}\right\},
\qquad k\geq p.
\end{equation}
We fix
\(
(\tau_{\max},\beta)=(0.15,156.25),\ (0.025,156.25),\
(0.00625,1562.5)
\)
for $\alpha=0.3,0.5,0.7$, respectively. The parameter $\beta$
remains fixed throughout each computation. The accepted adaptive
step is
\begin{equation}\label{eq:num_adaptive_clipping}
\tau_{k+1}
=\min\left\{\tau_{\max},\,1.16\tau_k,\,
\max\{\tau_{\min},\,0.86\tau_k,\,\widehat\tau_{k+1}\}\right\}.
\end{equation}
The final few steps are adjusted to reach
$T=10$ while preserving the same adjacent-step ratio bounds.

Figure~\ref{fig:num_adaptive_steps} shows the time steps on the three
meshes. The insets display the common startup and the subsequent
separation. The adaptive step sizes respond to changes in the solution at
later times. The graded step sizes increase according to the
prescribed mesh formula, while the fixed-step mesh remains nearly
uniform after the common startup.

Table~\ref{tab:num_adaptive_cost} reports the maximum errors
$E_N^{\max}$ defined in \eqref{eq:num_max_error}, the total numbers
of accepted steps $N$, and the elapsed times. The common startup
is included in every computation, and all meshes reach $t_N=10$.
We report the elapsed time for one run of the time-stepping
procedure in MATLAB R2024b. Reference evaluation, error
postprocessing, and plotting are excluded.

For all three values of $\alpha$, the adaptive meshes use fewer
steps and less time than the fixed-step meshes, with larger errors.
The modified graded meshes have comparable or lower step counts
and runtimes than the adaptive meshes, but larger errors than both
the adaptive and fixed-step meshes.

\section{Conclusion}\label{sect5}
In this work, we have established $L^2$ error estimates for the L2 method on nonuniform temporal meshes by developing a cross-row recombination technique. The key is to use the
parameter of a fixed time level to recombine all
preceding rows, while preserving positivity and monotonicity. These properties provide the structure
needed to close the induction in the proof of the error analysis. The resulting error estimates allow for the initial weak singularity of the exact solution and apply to meshes with nonmonotone time steps. On modified graded meshes, we obtained the global convergence rate \(N^{-\min\{r\alpha,3-\alpha\}}\), recovering the optimal order \(3-\alpha\) when \(r\ge(3-\alpha)/\alpha\). Numerical experiments support the theoretical convergence estimates.

\appendix

\section{Coefficient bounds and rational reductions}
\label{App-sec:sm_lemma_35}
For the ratios in \eqref{eq:r1} and \eqref{eq:r2}, the
coefficient definitions give the formulas below.
The diagonal coefficient is
\begin{equation}\label{App-eq:diagonal-coefficient-formula}
m_{k,k}=\frac{(\alpha\tau_{k-1}+2\tau_k)
 (\tau_{k-1}+\tau_k)^{-\alpha}}
 {(2-\alpha)(1-\alpha)\tau_k},\qquad k\ge2.
\end{equation}
The coefficient definitions and \eqref{App-eq:diagonal-coefficient-formula} give
\begin{equation}\label{App-eq:sm_r1}
\begin{aligned}
r_1^{(k)}(\alpha)
={}&\frac{\rho_k}{1+\rho_{k-1}}\Biggl[
\frac{[\alpha+2\rho_{k-1}(1+\rho_k)]
      [1+\rho_{k-1}(1+\rho_k)]}
     {\rho_{k-1}(\alpha+2\rho_k)}
\left(\frac{\rho_{k-1}(1+\rho_k)}
           {1+\rho_{k-1}(1+\rho_k)}\right)^\alpha\\
&\hspace{38mm}-[1+\rho_{k-1}+\rho_{k-1}(1+\rho_k)]
\Biggr],\qquad k\ge3,
\end{aligned}
\end{equation}
and, for the separate row $k=3$,
\begin{equation}\label{App-r2k3}
r_2^{(3)}(\alpha)
=\frac{\rho_3\rho_2(1+\rho_3)}{1+\rho_2}
\left[\rho_2+
\left(\frac{2(1-\alpha)}{\alpha+2\rho_3}-\rho_2\right)
\left(1+\frac1{\rho_2(1+\rho_3)}\right)^{1-\alpha}\right].
\end{equation}
For $k\ge4$, the corresponding expression is
\begin{equation}\label{App-r_2}
\begin{aligned}
r_2^{(k)}(\alpha)
={}&\frac{\rho_k\rho_{k-1}(1+\rho_k)}{1+\rho_{k-2}}
\Biggl[
\left(1+\frac{1+\rho_{k-2}}
 {\rho_{k-2}\rho_{k-1}(1+\rho_k)}\right)^{1-\alpha}
\frac{\alpha+2\rho_{k-2}[1+\rho_{k-1}(1+\rho_k)]}
 {\alpha+2\rho_k}\\
&\quad-
\left(\frac{1+\rho_{k-1}(1+\rho_k)}
 {\rho_{k-1}(1+\rho_k)}\right)^{1-\alpha}
\frac{(1+2\rho_{k-2}+\rho_{k-1}\rho_{k-2})
 [\alpha+2\rho_{k-1}(1+\rho_k)]}
 {(\alpha+2\rho_k)(1+\rho_{k-1})}\\
&\quad+\frac{(1+\rho_{k-2})\rho_{k-1}}{1+\rho_{k-1}}
\Biggr].
\end{aligned}
\end{equation}
\subsection{Strict concavity}
\label{App-sec:concavity}
For $F$ in \eqref{eq:F_definition}, we prove
$\partial_\alpha^2F<0$ on $[0,1]\times[0.83,1.17]^3$.
Combined with \eqref{eq:F_endpoints}, this gives $F\ge4\alpha/25$.

Differentiating the powers in \eqref{App-eq:sm_r1} and \eqref{App-r_2}
introduces the logarithms of two history-distance ratios:
\begin{equation}\label{App-eq:history-distance-logarithms}
\begin{aligned}
\ell_{k,1}
&:=\log\frac{t_k-t_{k-3}}{t_k-t_{k-2}}
=\log\frac{1+\rho_{k-1}(1+\rho_k)}
 {\rho_{k-1}(1+\rho_k)},\qquad k\ge3,\\
\ell_{k,2}
&:=\log\frac{t_k-t_{k-4}}{t_k-t_{k-2}}
=\log\left(1+\frac{1+\rho_{k-2}}
 {\rho_{k-2}\rho_{k-1}(1+\rho_k)}\right),\qquad k\ge4.
\end{aligned}
\end{equation}

In terms of these two logarithms,
twice differentiating \eqref{eq:F_definition} gives
\begingroup
\allowdisplaybreaks[1]
\begin{align}\label{App-eq:sm_F_second}
\frac{\partial^2F}{\partial\alpha^2}
={}&e^{-\alpha\ell_{k,1}}\Biggl\{
\frac{2\rho_k[1+\rho_{k-1}(1+\rho_k)]}
 {\rho_{k-1}(1+\rho_{k-1})}
\Biggl[\ell_{k,1}[2+(2-\alpha)\ell_{k,1}]\\*
&\quad+2[\rho_{k-1}(1+\rho_k)-\rho_k]
\Biggl(\frac{2(2-\alpha)}{(\alpha+2\rho_k)^3}
 +\frac2{(\alpha+2\rho_k)^2}
 +\frac{2(2-\alpha)\ell_{k,1}}{(\alpha+2\rho_k)^2}
 +\frac{2\ell_{k,1}}{\alpha+2\rho_k}
 +\frac{(2-\alpha)\ell_{k,1}^2}{\alpha+2\rho_k}
\Biggr)\Biggr]\notag\\
&\quad+
\frac{4\rho_k[1+\rho_{k-1}(1+\rho_k)]
 (1+2\rho_{k-2}+\rho_{k-1}\rho_{k-2})}
 {(1+\rho_{k-2})(1+\rho_{k-1})}\notag\\*
&\qquad\times\Biggl[\ell_{k,1}^2
 +2[\rho_{k-1}(1+\rho_k)-\rho_k]
 \left(\frac{\ell_{k,1}^2}{\alpha+2\rho_k}
       +\frac{2\ell_{k,1}}{(\alpha+2\rho_k)^2}
       +\frac2{(\alpha+2\rho_k)^3}\right)\Biggr]\notag\\
&\quad-
\frac{4\rho_k\rho_{k-1}(1+\rho_k)}{1+\rho_{k-2}}
 e^{(1-\alpha)\ell_{k,2}+\alpha\ell_{k,1}}\notag\\*
&\qquad\times\Biggl[\ell_{k,2}^2
 +2\bigl[\rho_{k-2}[1+\rho_{k-1}(1+\rho_k)]-\rho_k\bigr]
 \left(\frac{\ell_{k,2}^2}{\alpha+2\rho_k}
       +\frac{2\ell_{k,2}}{(\alpha+2\rho_k)^2}
       +\frac2{(\alpha+2\rho_k)^3}\right)\Biggr]
\Biggr\}-2.\notag
\end{align}
\endgroup
Both logarithms are positive. Throughout the parameter domain,
\[
\rho_{k-1}(1+\rho_k)-\rho_k>0,
\qquad
\rho_{k-2}[1+\rho_{k-1}(1+\rho_k)]-\rho_k>0,
\]
and every denominator is positive. Thus the square-bracketed
expressions increase with their nonnegative logarithmic arguments.

For $\varrho>1$, the logarithmic series and a geometric bound
for its tail give $L_-(\varrho)\le\log\varrho\le L_+(\varrho)$, where

\begin{equation}\label{App-eq:sm_log_bounds}
\begin{aligned}
L_-(\varrho)&:=2\left[\frac{\varrho-1}{\varrho+1}
 +\frac13\left(\frac{\varrho-1}{\varrho+1}\right)^3\right], \qquad
L_+(\varrho):=L_-(\varrho)+
\frac{2\left(\dfrac{\varrho-1}{\varrho+1}\right)^5}
 {5\left[1-\left(\dfrac{\varrho-1}{\varrho+1}\right)^2\right]}.
\end{aligned}
\end{equation}
Using
\[
a^{1-\alpha}b^\alpha
\ge \left(\frac{1-\alpha}{a}+\frac{\alpha}{b}\right)^{-1},
\qquad a,b>0,\quad 0\le\alpha\le1,
\]
with $a=e^{\ell_{k,2}}$ and $b=e^{\ell_{k,1}}$,
together with the preceding logarithmic bounds, we obtain
\begin{equation}\label{App-eq:concavity-rational-replacements}
\begin{aligned}
\ell_{k,1}
&\le
L_+\left(\frac{1+\rho_{k-1}(1+\rho_k)}
 {\rho_{k-1}(1+\rho_k)}\right),\\
\ell_{k,2}
&\ge
L_-\left(1+\frac{1+\rho_{k-2}}
 {\rho_{k-2}\rho_{k-1}(1+\rho_k)}\right),\\
e^{(1-\alpha)\ell_{k,2}+\alpha\ell_{k,1}}
&\ge
\frac{[\rho_{k-2}\rho_{k-1}(1+\rho_k)+1+\rho_{k-2}]
 [1+\rho_{k-1}(1+\rho_k)]}
 {\rho_{k-1}(1+\rho_k)
  [\alpha+\rho_{k-2}[1+\rho_{k-1}(1+\rho_k)]]}.
\end{aligned}
\end{equation}
Define $\mathcal U_k$ from the expression inside the outer braces
in \eqref{App-eq:sm_F_second} by replacing the logarithms in the two
positive terms by the upper bound in
\eqref{App-eq:concavity-rational-replacements}, and replacing the
logarithm and the entire exponential factor in the subtracted
term by their respective lower bounds.
Then $\mathcal U_k$ is a rational function of
$(\alpha,\rho_k,\rho_{k-1},\rho_{k-2})$.
The preceding monotonicity and the arithmetic mean inequality give
\begin{equation}\label{App-eq:sm_power_upper}
\partial_\alpha^2F\le e^{-\alpha\ell_{k,1}}\mathcal U_k-2,
\qquad e^{-\alpha\ell_{k,1}}
\le1-\frac{\alpha}{1+\rho_{k-1}(1+\rho_k)}.
\end{equation}
If $\mathcal U_k\le0$, the second derivative is at most $-2$.
If $\mathcal U_k>0$, \eqref{App-eq:sm_power_upper} applies to its
positive multiplier. It therefore suffices to verify
\begin{equation}\label{App-eq:sm_rational_goal}
2-\left(1-\frac{\alpha}{1+\rho_{k-1}(1+\rho_k)}\right)
\mathcal U_k>0.
\end{equation}
A positive common denominator of the left-hand side is
\begin{equation}\label{App-eq:sm_denominator}
\begin{aligned}
D_{\mathrm{conc}}
={}&3600[\rho_{k-1}(1+\rho_k)]^2\rho_{k-1}
 (\alpha+2\rho_k)^3
 [1+\rho_{k-1}(1+\rho_k)]^2\\
&\times[1+2\rho_{k-1}(1+\rho_k)]^6
 (1+\rho_{k-1})(1+\rho_{k-2})
 [\alpha+\rho_{k-2}[1+\rho_{k-1}(1+\rho_k)]]\\
&\times[1+\rho_{k-2}
       +2\rho_{k-2}\rho_{k-1}(1+\rho_k)]^6.
\end{aligned}
\end{equation}
Appendix~\ref{App-sec:bernstein-certificates} certifies the polynomial
obtained by multiplying \eqref{App-eq:sm_rational_goal} by
$D_{\mathrm{conc}}$, including $3600$, without cancellation.
This proves strict concavity and, by \eqref{eq:F_endpoints},
$F(\alpha,\boldsymbol\rho)\ge4\alpha/25$.

\subsection{Scalar estimates for cross-row recombination}
\label{App-sec:cross-row-lemma}
\begin{lemma}[Scalar estimates]\label{App-lem:cross-row-recombination}
Let $0<\alpha<1$ and $\rho_j\in[0.86,1.16]$ for $j\ge2$.
For every fixed $k\ge3$ and $2\le n\le k$,
\begin{subequations}\label{App-eq:four-cross-row-margins}
\begin{align}
1-2\eta_n+\eta_k&\ge\frac15,\label{App-eq:cross-L1}\\
2\eta_n-\eta_k&\ge\frac{\alpha}{2},\label{App-eq:cross-L2}\\
2\eta_n-\eta_k-\frac{m_{n-1,n-1}}{m_{n,n}}\eta_k^2
&\ge\frac{\alpha}{10}.\label{App-eq:cross-L4}
\end{align}
\end{subequations}
\end{lemma}

\begin{proof}
For $n\ge3$, use $\eta_n=(1-r_1^{(n)})/2$ from \eqref{eta}
and the weighted harmonic and arithmetic mean bounds
\[
\frac{\rho_{n-1}(1+\rho_n)}{\alpha+\rho_{n-1}(1+\rho_n)}
\le
\left(\frac{\rho_{n-1}(1+\rho_n)}{1+\rho_{n-1}(1+\rho_n)}\right)^\alpha
\le1-\frac{\alpha}{1+\rho_{n-1}(1+\rho_n)}.
\]
Using \eqref{App-eq:sm_r1} in \eqref{eta}, these power bounds yield
the rational lower and upper bounds for $\eta_n$:
\begin{align}
\eta_n^{1,-}
&:=\frac12\Biggl\{1-\frac{\rho_n}{1+\rho_{n-1}}
\Biggl[
\frac{[\alpha+2\rho_{n-1}(1+\rho_n)]
[1+\rho_{n-1}(1+\rho_n)-\alpha]}
{\rho_{n-1}(\alpha+2\rho_n)}
\notag\\[-1mm]
&\hspace{43mm}-[1+\rho_{n-1}+\rho_{n-1}(1+\rho_n)]\Biggr]\Biggr\},
\label{App-eq:supp-eta-H-minus}\\
\eta_n^{1,+}
&:=\frac12\Biggl\{1-\frac{\rho_n}{1+\rho_{n-1}}
\Biggl[
\frac{[\alpha+2\rho_{n-1}(1+\rho_n)]
[1+\rho_{n-1}(1+\rho_n)](1+\rho_n)}
{(\alpha+2\rho_n)[\alpha+\rho_{n-1}(1+\rho_n)]}
\notag\\[-1mm]
&\hspace{43mm}-[1+\rho_{n-1}+\rho_{n-1}(1+\rho_n)]\Biggr]\Biggr\}.
\label{App-eq:supp-eta-H-plus}
\end{align}
Thus
\begin{equation}\label{App-eq:eta-harmonic-bracket}
\eta_n^{1,-}\le\eta_n\le\eta_n^{1,+},\qquad n\ge3.
\end{equation}
The superscript $1$ denotes these mean bounds; replacing $n$ by $k$
defines $\eta_k^{1,\pm}$.
For the diagonal coefficients, the arithmetic mean bound also gives
\begin{equation}\label{App-eq:diagonal-ratio-upper-main}
\begin{aligned}
\frac{m_{n-1,n-1}}{m_{n,n}}
&=\frac{\rho_n(\alpha+2\rho_{n-1})}{\rho_{n-1}(\alpha+2\rho_n)}
\left(\frac{\rho_{n-1}(1+\rho_n)}{1+\rho_{n-1}}\right)^\alpha\le\frac{\rho_n(\alpha+2\rho_{n-1})}{\rho_{n-1}(\alpha+2\rho_n)}
\left[1-\alpha+\frac{\alpha\rho_{n-1}(1+\rho_n)}{1+\rho_{n-1}}\right].
\end{aligned}
\end{equation}
Appendix~\ref{App-sec:bernstein-certificates} proves positivity of
the following rational expressions for $n,k\ge3$:
\begin{align}
&1-2\eta_n^{1,+}+\eta_k^{1,-}-\frac15,
\label{App-eq:supp-target-ordinary-L1}\\
&\frac{2\eta_n^{1,-}-\eta_k^{1,+}}{\alpha}-\frac12,
\label{App-eq:supp-target-ordinary-L2}\\
&\frac1\alpha\Biggl\{2\eta_n^{1,-}-\eta_k^{1,+}
-\frac{\rho_n(\alpha+2\rho_{n-1})}{\rho_{n-1}(\alpha+2\rho_n)}
\left[1-\alpha+\frac{\alpha\rho_{n-1}(1+\rho_n)}{1+\rho_{n-1}}\right]
(\eta_k^{1,+})^2\Biggr\}-\frac1{10}.
\label{App-eq:supp-target-ordinary-L4}
\end{align}
Taking $n=k$ in \eqref{App-eq:supp-target-ordinary-L2} gives
$0<\alpha/2\le\eta_k^{1,-}\le\eta_k\le\eta_k^{1,+}$,
so $\eta_k^2\le(\eta_k^{1,+})^2$.
Together with \eqref{App-eq:eta-harmonic-bracket} and
\eqref{App-eq:diagonal-ratio-upper-main}, the three certified
inequalities prove the assertions for $n\ge3$.

For $n=2$, we instead use the explicit startup formulas
\begin{equation}\label{App-eq:eta2-main}
\eta_2=\frac{\alpha(1+\rho_2)^2}{2(\alpha+2\rho_2)},\qquad
\frac{m_{1,1}}{m_{2,2}}
=\frac{(2-\alpha)\rho_2(1+\rho_2)^\alpha}{\alpha+2\rho_2}
\le\frac{(2-\alpha)\rho_2(1+\alpha\rho_2)}{\alpha+2\rho_2}.
\end{equation}
The corresponding certified positive expressions are
\begin{align}
&1-2\eta_2+\eta_k^{1,-}-\frac15,
\label{App-eq:supp-target-startup-L1}\\
&\frac{2\eta_2-\eta_k^{1,+}}{\alpha}-\frac12,
\label{App-eq:supp-target-startup-L2}\\
&\frac1\alpha\left[2\eta_2-\eta_k^{1,+}
-\frac{(2-\alpha)\rho_2(1+\alpha\rho_2)}{\alpha+2\rho_2}
(\eta_k^{1,+})^2\right]-\frac1{10}.
\label{App-eq:supp-target-startup-L4}
\end{align}
Using \eqref{App-eq:eta2-main} and the same bounds for $\eta_k$
proves the three assertions for $n=2$.
\end{proof}

\subsection{Variation of \texorpdfstring{$\eta$}{eta} and the row \texorpdfstring{$n=3$}{n=3}}
\label{App-sec:eta-Taylor-bounds}
For the variation estimate, sharper bounds for the power in
\eqref{App-eq:sm_r1} follow from quadratic and cubic Taylor bounds.
For $j\ge3$, define
\begin{align}
\mathscr E_-^{(j)}
&:=1-\alpha L_+\!\left(1+\frac1{\rho_{j-1}(1+\rho_j)}\right)
+\frac{\alpha^2}{2}L_+^2\!\left(1+\frac1{\rho_{j-1}(1+\rho_j)}\right)
\notag\\
&\qquad-\frac{\alpha^3}{6}L_+^3\!\left(1+\frac1{\rho_{j-1}(1+\rho_j)}\right),
\label{App-eq:supp-power-Taylor-minus}\\
\mathscr E_+^{(j)}
&:=1-\alpha L_-\!\left(1+\frac1{\rho_{j-1}(1+\rho_j)}\right)
+\frac{\alpha^2}{2}L_-^2\!\left(1+\frac1{\rho_{j-1}(1+\rho_j)}\right).
\label{App-eq:supp-power-Taylor-plus}
\end{align}
These rational expressions satisfy
\begin{equation}\label{App-eq:supp-power-Taylor-order}
\mathscr E_-^{(j)}
\le\left(\frac{\rho_{j-1}(1+\rho_j)}{1+\rho_{j-1}(1+\rho_j)}\right)^\alpha
\le\mathscr E_+^{(j)}.
\end{equation}
Indeed, $1-v+v^2/2-v^3/6\le e^{-v}\le1-v+v^2/2$ for $v\ge0$.
The decreasing cubic and $L_+$ give the lower bound; applying
the quadratic bound to $e^{-\alpha L_-}$ gives the upper bound.
Substitution into \eqref{App-eq:sm_r1} and use of
\eqref{eta} yield
\begin{equation}
\label{App-eq:supp-eta-T-minus}
\begin{aligned}
\eta_j^{2,\pm}
:={}&\frac12\Biggl\{1-\frac{\rho_j}{1+\rho_{j-1}}
\Biggl[
\frac{[\alpha+2\rho_{j-1}(1+\rho_j)][1+\rho_{j-1}(1+\rho_j)]}
{\rho_{j-1}(\alpha+2\rho_j)}\mathscr E_\mp^{(j)}\\
&\hspace{32mm}-[1+\rho_{j-1}+\rho_{j-1}(1+\rho_j)]\Biggr]\Biggr\}.
\end{aligned}
\end{equation}
Here $2$ denotes the Taylor bounds, with $\mp$ opposite to $\pm$; hence
\begin{equation}\label{App-eq:supp-eta-T-bracket}
\eta_j^{2,-}\le\eta_j\le\eta_j^{2,+},\qquad j\ge3.
\end{equation}
Appendix~\ref{App-sec:bernstein-certificates} verifies that
\begin{align}
&\frac{\eta_j^{2,-}}{\alpha}
-\left(1-\frac{\alpha}{3}-\frac{\alpha(1-\alpha)}{10}-\frac2{25}\right),
\label{App-eq:supp-target-eta-lower}\\
&1-\frac{\alpha}{3}-\frac{\alpha(1-\alpha)}{10}+\frac2{25}
-\frac{\eta_j^{2,+}}{\alpha}
\label{App-eq:supp-target-eta-upper}
\end{align}
are positive. Consequently,
\begin{equation}\label{App-eq:eta-band-main}
\alpha\left(1-\frac{\alpha}{3}-\frac{\alpha(1-\alpha)}{10}-\frac2{25}\right)
\le\eta_j\le
\alpha\left(1-\frac{\alpha}{3}-\frac{\alpha(1-\alpha)}{10}+\frac2{25}\right),
\qquad j\ge3.
\end{equation}
The two bounds differ by $4\alpha/25$, proving
$|\eta_k-\eta_n|\le4\alpha/25$ for $n,k\ge3$.

They also give the uniform estimates
\begin{equation}\label{App-eq:eta-uniform-main}
\frac{337\alpha}{600}\le\eta_j\le\frac{56}{75}<\eta_*
\quad(j\ge3),\qquad
0<\eta_2\le\frac{1458}{2075}<\eta_*.
\end{equation}
Indeed, the lower bracket in \eqref{App-eq:eta-band-main} is at least
$1-1/3-1/40-2/25=337/600$, and the upper bound does not exceed
$\alpha(27/25-\alpha/3)\le56/75$.
Increasing $\alpha$ to $1$ in \eqref{App-eq:eta2-main} gives
$\eta_2\le(1+\rho_2)^2/[2(1+2\rho_2)]$; this bound increases
with $\rho_2$ and is at most $1458/2075$ for $\rho_2\le1.16$.

For $F_3$, the factor $2(1-\alpha)/(\alpha+2\rho_3)-\rho_2$
in \eqref{App-r2k3} may have either sign.
Multiplying \eqref{App-eq:supp-power-Taylor-order} by
$1+1/[\rho_2(1+\rho_3)]$ gives an upper bound for $r_2^{(3)}$
using $\mathscr E_+^{(3)}$ when this factor is nonnegative and
$\mathscr E_-^{(3)}$ otherwise:
\begin{equation}\label{App-eq:supp-r2-three-plus}
\begin{aligned}
r_{2,\mathrm{up}}^{(3),\pm}
:={}&\frac{\rho_3\rho_2(1+\rho_3)}{1+\rho_2}
\Biggl[\rho_2+
\left(\frac{2(1-\alpha)}{\alpha+2\rho_3}-\rho_2\right)
\frac{1+\rho_2(1+\rho_3)}{\rho_2(1+\rho_3)}
\mathscr E_\pm^{(3)}\Biggr].
\end{aligned}
\end{equation}
Since $r_1^{(3)}=1-2\eta_3\ge1-2\eta_3^{2,+}$, in each
of the corresponding sign cases we obtain
\begin{equation}\label{App-eq:supp-F3-bounds}
F_3\ge F_{3,\mathrm{low}}^\pm
:=(4-2\alpha)(1-2\eta_3^{2,+})
-4r_{2,\mathrm{up}}^{(3),\pm}+2\alpha-\alpha^2.
\end{equation}
Both rational expressions
\begin{equation}\label{App-eq:supp-target-F3}
\frac{F_{3,\mathrm{low}}^+}{\alpha}-\frac4{25},
\qquad
\frac{F_{3,\mathrm{low}}^-}{\alpha}-\frac4{25}
\end{equation}
are positive on the whole domain $[0,1]\times[0.83,1.17]^2$,
as verified below. Applying the appropriate bound in each sign
case proves $F_3\ge4\alpha/25$.

\section{Exact Bernstein certificates}
\label{App-sec:bernstein-certificates}
\label{App-sec:Bernstein-coefficients}
All coefficients below are computed with exact rational arithmetic. Ratios in $[a,b]$ are normalized by
$\rho=a+(b-a)\xi$, $0\le\xi\le1$, while $\xi_0=\alpha$.
For a polynomial in ${\xi}=(\xi_0,\ldots,\xi_s)$ of
multidegree ${d}$, the power and
Bernstein representations are related by
\begin{equation}\label{App-eq:supp-general-Bernstein-transform}
\begin{aligned}
P({\xi})
&=\sum_{0\le{\nu}\le{d}}
c_{{\nu}}{\xi}^{{\nu}}
=\sum_{0\le{i}\le{d}}
b_{{i}}\prod_{j=0}^{s}
\binom{d_j}{i_j}\xi_j^{i_j}(1-\xi_j)^{d_j-i_j},
\qquad
b_{{i}}
=\sum_{0\le{\nu}\le{i}}
c_{{\nu}}\prod_{j=0}^{s}
\frac{\binom{i_j}{\nu_j}}{\binom{d_j}{\nu_j}}.
\end{aligned}
\end{equation}
Inequalities between multi-indices are componentwise.
The Bernstein basis functions are nonnegative and sum to one.
Hence $\min_i b_i\le P(\xi)\le\max_i b_i$ for all
$\xi\in[0,1]^{s+1}$, so a positive minimum coefficient
ensures that $P$ is positive throughout this domain.

For concavity, multiply \eqref{App-eq:sm_rational_goal} by
$D_{\mathrm{conc}}$ in \eqref{App-eq:sm_denominator}, including
$3600$, without cancellation.
Normalize the ratios in the order $(\rho_k,\rho_{k-1},\rho_{k-2})$
using $a=83/100$ and $b=117/100$.
The resulting polynomial has multidegree $(6,21,19,8)$ in
$(\alpha,\xi_1,\xi_2,\xi_3)$.
All $27{,}720$ coefficients are positive, with minimum
\begin{equation}\label{App-eq:sm_minimum}
b_{4,0,0,0}>97923459813770>0
\qquad
\bigl(b_{4,0,0,0}\approx9.79234598137709\times10^{13}\bigr).
\end{equation}

Substitute all definitions and reduce each of the ten scalar
expressions to $N/D$ with integer coefficients.
Normalize the combined coefficient greatest common divisor to one,
and choose the common sign so that $D$ is positive at the
midpoint of the corresponding parameter domain.
Values at $\alpha=0$ are defined by continuity.
Throughout, $\alpha\in[0,1]$ and all ratios lie in $[0.86,1.16]$,
except in \eqref{App-eq:supp-target-F3}, where they lie in $[0.83,1.17]$.
The variable orders are
$(\alpha,\rho_{n-1},\rho_n,\rho_{k-1},\rho_k)$ for
\eqref{App-eq:supp-target-ordinary-L1}--\eqref{App-eq:supp-target-ordinary-L4};
$(\alpha,\rho_2,\rho_{k-1},\rho_k)$ for
\eqref{App-eq:supp-target-startup-L1}--\eqref{App-eq:supp-target-startup-L4};
$(\alpha,\rho_{j-1},\rho_j)$ for
\eqref{App-eq:supp-target-eta-lower} and \eqref{App-eq:supp-target-eta-upper};
and $(\alpha,\rho_2,\rho_3)$ for \eqref{App-eq:supp-target-F3}.
After affine normalization, \eqref{App-eq:supp-general-Bernstein-transform}
is applied to both numerator and denominator.
All denominator coefficients are strictly positive.
Table~\ref{App-tab:supp-certificate-summary} gives the numerator data;
its indices use the normalized variables in the orders above.

\begin{table}[htbp]
\centering
\small
\setlength{\tabcolsep}{4pt}
\caption{Bernstein coefficient data for the normalized numerator
polynomials. The minimum coefficients are reported approximately;
their positivity is verified using exact arithmetic.}
\label{App-tab:supp-certificate-summary}
\begin{tabular}{@{}lccccc@{}}
\hline
Expression & Multidegree & Count & Nonpositive
& \shortstack{Minimum\\index} & \shortstack{Minimum\\coefficient}\\
\hline
\eqref{App-eq:supp-target-ordinary-L1}
& \((4,2,3,2,2)\) & \(540\) & \(0\)
& \((4,0,3,0,0)\) & \(2.353582990\times10^{1}\)\\
\eqref{App-eq:supp-target-ordinary-L2}
& \((3,2,2,2,3)\) & \(432\) & \(0\)
& \((0,0,1,0,0)\) & \(4.744346169\times10^{-1}\)\\
\eqref{App-eq:supp-target-ordinary-L4}
& \((5,2,2,4,6)\) & \(1890\) & \(0\)
& \((0,0,0,0,0)\) & \(6.294892570\times10^{2}\)\\
\eqref{App-eq:supp-target-startup-L1}
& \((3,2,2,2)\) & \(108\) & \(0\)
& \((3,2,0,0)\) & \(4.428076704\times10^{0}\)\\
\eqref{App-eq:supp-target-startup-L2}
& \((3,2,2,3)\) & \(144\) & \(0\)
& \((0,0,0,0)\) & \(6.150987181\times10^{0}\)\\
\eqref{App-eq:supp-target-startup-L4}
& \((5,2,4,6)\) & \(630\) & \(0\)
& \((0,0,0,0)\) & \(6.751226672\times10^{2}\)\\
\eqref{App-eq:supp-target-eta-lower}
& \((3,8,8)\) & \(324\) & \(0\)
& \((1,0,0)\) & \(2.785587141\times10^{5}\)\\
\eqref{App-eq:supp-target-eta-upper}
& \((3,16,16)\) & \(1156\) & \(0\)
& \((0,0,0)\) & \(9.949499105\times10^{11}\)\\
\eqref{App-eq:supp-target-F3}, \(F_{3,\mathrm{low}}^+\)
& \((4,16,16)\) & \(1445\) & \(0\)
& \((4,0,0)\) & \(7.988977703\times10^{11}\)\\
\eqref{App-eq:supp-target-F3}, \(F_{3,\mathrm{low}}^-\)
& \((4,16,16)\) & \(1445\) & \(0\)
& \((4,0,16)\) & \(5.201161800\times10^{10}\)\\
\hline
\end{tabular}
\end{table}

The smallest of these ten positive minima corresponds to
\eqref{App-eq:supp-target-ordinary-L2} and has the exact value
\(
\frac{18532602222867}{39062500000000}>\frac{47}{100}
\).
Since the denominator Bernstein coefficients are also strictly
positive, all ten rational expressions are positive on their
respective parameter domains. This completes the scalar estimates,
the row-to-row variation bound, and the proof for $n=3$.

\section*{Acknowledgements}
 C. Quan is supported by the National Natural Science Foundation of China (Grant No. 12622124) and Hetao Shenzhen-Hong Kong Science and Technology  Innovation Cooperation Zone Project (No. HZQSWS-KCCYB-2024016). The work of J. Yang is supported by the National Science Foundation of China (No.12271240, 12426312) and the Shenzhen Natural Science Fund (RCJC20210609103819018).


\section*{Declarations}
The authors used GPT-6 Astra (OpenAI)
through Codex with the Ultra reasoning setting to assist with language
polishing and checking the associated
Mathematica code for verifying polynomial
positivity using Bernstein coefficients.
The authors assume responsibility for all content.

\bibliography{bibfile}
\bibliographystyle{amsplain}

\end{document}